\RequirePackage{fix-cm}
\documentclass[smallextended]{svjour3}       
\smartqed  
\usepackage{graphicx}
\usepackage{latexsym}

\usepackage{amsmath,amssymb,amsfonts,fullpage,amssymb,amsthm}

\usepackage{verbatim}
\usepackage[usenames]{color}
\usepackage{hyperref}
\usepackage{url}
\usepackage{mathrsfs}
\usepackage{tikz,tikz-qtree,ifthen}
\usepackage{float}
\usepackage{centernot}
\newcommand{\bk}{\mathbf{k}}

\newcommand{\Fin}{\mathrm{Fin}}
\newcommand{\fin}{\mathrm{fin}}
\newcommand{\FIN}{\mathrm{FIN}}
\newcommand{\depth}{\mathrm{depth}}
\newcommand{\supp}{\mathrm{supp}}
\newcommand{\dom}{\mathrm{dom}}

\newcommand*{\lex}{<_{\textnormal{lex}}}

\newcommand*{\omle}[1]{\omega^{\protect \centernot \downarrow \le #1}}

\newcommand*{\ome}[1]{\omega^{\protect \centernot \downarrow #1}}

\newcommand*{\set}[1]{\left\{#1\right\}}

\newcommand{\om}{\omega}
\newcommand{\sse}{\subseteq}
\newcommand{\re}{\upharpoonright}
\newcommand{\al}{\alpha}
\newcommand{\Fraisse}{Fra{\"{i}}ss{\'{e}}}
\newcommand{\rgl}{\rangle}
\newcommand{\lgl}{\langle}
\newcommand{\ra}{\rightarrow}

\journalname{Archive for Mathematical Logic}
\title{Ramsey degrees of ultrafilters,
pseudointersection numbers,
and  the tools of   topological Ramsey spaces}
 \thanks{Dobrinen was supported by National Science Foundation Grants  DMS-1600781 and DMS-1901753. Navarro Flores was supported by CONACYT, and Dobrinen's National Science Foundation Grant 
 DMS-1600781.}

\author{Natasha Dobrinen \and
        Sonia Navarro Flores 
}

\institute{Natasha Dobrinen \at
              Department of Mathematics, University of Denver, Denver, USA \\
              \email{natasha.dobrinen@du.edu}           
           \and
           Sonia Navarro \at
              Unidad Morelia,
Universidad  Nacional Aut\'{o}noma de M\'{e}xico, Morelia, Mexico\\
              \email{sonian@matmor.unam.mx}
}
\date{Received: date / Accepted: date}

\begin{document}

\maketitle
\begin{abstract}
This paper investigates properties of $\sigma$-closed forcings which generate ultrafilters satisfying weak partition relations. 
The Ramsey degree  of an ultrafilter $\mathcal{U}$ for $n$-tuples,  denoted $t(\mathcal{U},n)$,  is the  smallest number $t$ such that given any $l\ge 2$ and  coloring $c:[\om]^n\ra  l$, there is a member $X\in\mathcal{U}$ such that the restriction of $c$ to $[X]^n$ has no more than $t$ colors. 
Many well-known  $\sigma$-closed  forcings are known to  generate ultrafilters with finite Ramsey degrees, but finding the precise degrees can sometimes prove elusive  or quite involved,
at best. 
In this paper, we utilize methods of topological Ramsey spaces  to 
calculate Ramsey degrees   of several classes of ultrafilters generated by $\sigma$-closed forcings.
These include a hierarchy of forcings due to Laflamme which generate weakly Ramsey and weaker rapid p-points,   forcings of Baumgartner and Taylor and of Blass and  generalizations,  and  the collection of non-p-points generated by the forcings 
$\mathcal{P}(\om^k)/\Fin^{\otimes k}$. 
We provide a general approach to calculating the Ramsey degrees of these ultrafilters, obtaining 
 new results as well as   streamlined   proofs of  previously known results. 
In the second half of the paper, 
 we  calculate pseudointersection and tower numbers for these $\sigma$-closed forcings   and their relationships with the  classical  pseudointersection number $\mathfrak{p}$.
\end{abstract}


\section{Introduction}\label{intro}

Ramsey's Theorem states that given any $n,l\ge 2$ and a coloring $c:[\om]^n\ra l$, there is an infinite subset $X\sse \om$ such that $c$ is constant on $[X]^n$ (\cite{Ramsey30}). 
A {\em Ramsey ultrafilter} is an ultrafilter satisfying instances of Ramsey's Theorem. 
Thus, an ultrafilter $\mathcal{U}$ on base set $\om$ is {\em Ramsey} if and only if  given any  $U\in\mathcal{U}$, 
$n,l\ge 2$, and coloring $c:[U]^n\ra l$, there is some $X\in\mathcal{U}$ with $X\sse U$ such that $c$ is constant on $[X]^n$.
This is denoted by 
\begin{equation}
\mathcal{U}\ra(\mathcal{U})^n_l.
\end{equation}
Ramsey ultrafilters can be constructed using CH, or just  
cov$(\mathcal{M})=\mathfrak{c}$.
They can also be constructed by forcing with 
 $\mathcal{P}(\om)/\Fin$, or equivalently, $([\om]^{\om},\sse^*)$.
Generalizations of this forcing in many different directions  have been used to construct ultrafilters which are not Ramsey but have similar properties in the following sense.

  \begin{definition}\label{defn.Ramseydegree}
  Let  $\mathcal{U}$ be an ultrafilter on a countable base set $S$.
 For $n\ge 2$,   define 
 $t(\mathcal{U},n)$
 to be the least  number $t$, if it exists,  such that for each $l\ge 2$ and each coloring $c:[S]^n\ra l$,
 there is a member $X\in \mathcal{U}$ such that 
 the restriction of $c$ to $[X]^n$ takes no more than $t$ colors.
When $ t(\mathcal{U},n)$ exists,  it 
is called the {\em Ramsey degree} of $\mathcal{U}$ for $n$-tuples, and 
we write 
 $$
 \mathcal{U}\ra(\mathcal{U})^n_{l,t(\mathcal{U},n)}.
 $$
  \end{definition}

A weakly Ramsey ultrafilter is one satisfying $ t(\mathcal{U},2)=2$. 
A plethora of 
ultrafilters with  finite Ramsey degrees have been forced by various $\sigma$-closed posets in the literature. 
The following examples are indicative of the variety of  such posets. 
 Laflamme in \cite{Laflamme} forced a hierarchy of rapid p-points above a weakly Ramsey ultrafilter.
  Baumgartner and Taylor in \cite{Baumgartner/Taylor78} forced  $k$-arrow, not $(k+1)$-arrow ultrafilters, for each $k\ge 2$.  These are rapid p-points  satisfying an asymmetric partition relation.
  In   \cite{Blass73},
   Blass  constructed a p-point which has two Rudin-Keisler incomparable p-points below it, using a sort of two-dimensional forcing. 
  The forcing $\mathcal{P}(\om\times\om)/\Fin^{\otimes 2}$ was investigated by  Szyma\'{n}ski and Zhou in \cite{Szymanski/Zhou81} and shown to produce an ultrafilter, denoted $\mathcal{G}_2$,  which is not a p-point but has Ramsey degree $t(\mathcal{G}_2,2)=4$. 
  This ultrafilter was shown to be a weak p-point in \cite{Blass/Dobrinen} and  investigations of its Tukey type are included in that paper. 
Further extensions of $\mathcal{P}(\om)/\Fin$ to finite  dimensions  ($\mathcal{P}(\om^k)/\Fin^{\otimes k}$ for $2\le k<\om$) were  investigated by Kurili\'{c} in \cite{Kurilic15} and by
   Dobrinen in \cite{highDimensional}.
   In fact, the  natural hierarchy of forcings $\mathcal{P}(\om^{\al})/\Fin^{\otimes \al}$,  for all countable ordinals $\al$,  was shown to be forcing equivalent to certain topological Ramsey spaces in \cite{DobrinenJML16}.
   These and other ultrafilters will be investigated in this paper in two directions:
   First, we will find a general method for calculating Ramsey degrees of ultrafilters from these classes. 
   Second, we will investigate the pseudointersection and tower numbers of the forcings which generate these ultrafilters.

In recent years, the use of topological Ramsey spaces to investigate forcings generating ultrafilters has provided means for obtaining results 
which remained elusive when simply using the original forcings.
This began in \cite{R1}, where Dobrinen and Todorcevic constructed  a Ramsey space dense inside of the  forcing of  Laflamme in \cite{Laflamme} 
which produces a weakly Ramsey, not Ramsey ultrafilter,  denoted $\mathcal{U}_1$,
 in order to calculate the exact Rudin-Keisler and Tukey structures below this ultrafilter.
 This idea was extended in 
  \cite{Rn}, 
\cite{fraisseClasses}, \cite{highDimensional},
and 
\cite{DobrinenJML16}, providing new collections of topological Ramsey spaces dense in known forcings, 
such as those of Baumgartner-Taylor, Blass, Laflamme, and Szyma\'{n}ski-Zhou mentioned above,  
as well as creating new forcings which produce ultrafilters with interesting partition relations. 
Such Ramsey spaces were used to find exact Rudin-Keisler and Tukey structures below those ultrafilters. 
An overview of this area can be found in 
\cite{Dobrinen/Seals}.

Once constructed, it turned  out that  the topological Ramsey space structure of these forcings can be used to investigate how closely these ultrafilters resemble a Ramsey ultrafilter.
One such investigation was recently carried out by 
 Dobrinen and Hathaway in 
\cite{Dobrinen/Hathaway19}, where they show that  each of  the ultrafilters mentioned above have properties similar to that of a Ramsey ultrafilter in the sense of the barren extensions of Henle, Mathias, and Woodin in \cite{Henle/Mathias/Woodin85}.
In this paper, we investigate properties of the same ultrafilters in \cite{Dobrinen/Hathaway19}, utilizing topological Ramsey space techniques to better handle the properties of the forcings.

  Background  on classical results is provided in Section \ref{sec:Background}  and 
  an overview of topological Ramsey spaces and their associated ultrafilters appears in Section \ref{sec:tRs}.
While there is not space to reproduce all information on the  topological Ramsey spaces addressed in  this paper,
an overview of their salient properties to aid the reader  appear in Section \ref{sec:tRs&uf}.
 These include the spaces mentioned above as well as the Carlson-Simpson infinite dual Ramsey space and the space $\FIN^{[\infty]}$ of infinite block sequences of Milliken and related spaces $\FIN_k^{[\infty]}$  of Todorcevic, building on work of Gowers.

The first focus of this paper is on finding exact Ramsey degrees of  the aforementioned ultrafilters from the papers  
\cite{Laflamme}, \cite{Baumgartner/Taylor78},   \cite{Blass73}, \cite{Szymanski/Zhou81}, \cite{Kurilic15},  and \cite{highDimensional},
 as well as new forcings from \cite{fraisseClasses}.
Finding the Ramsey degrees for these ultrafilters based on the original forcings is not always simple.
A case in point is Laflamme's proof that the Ramsey degrees for the weakly Ramsey ultrafilter $\mathcal{U}_1$ are  
$t(\mathcal{U}_1,n)= 2^{n-1}$, for $n\ge 2$.
His proof can be greatly simplified by utilizing the structure of the topological Ramsey space dense in his forcing.
In Section \ref{sec:degrees}, we provide a simplified proof of this result, and provide  streamlined proofs of  other  results 
stated without proof in 
 \cite{Laflamme} (see Theorem \ref{thm.RamseydegLaflamme}).
 In Theorem \ref{numbers}, 
  we provide a general approach for determining Ramsey degrees for ultrafilters forced by Ramsey spaces with certain natural properties.
  Then in Subsection
\ref{subsec.5.2}, we apply this to find Ramsey degrees for the ultrafilters of Laflamme, Baumgartner-Taylor,  Blass, and the hypercube spaces from \cite{fraisseClasses}.

In Section  \ref{sec:hdegrees} we calculate the  Ramsey degrees $t(\mathcal{G}_k,2)$, where $\mathcal{G}_k$ is the ultrafilters forced by $\mathcal{P}(\om^k)/\Fin^{\otimes k}$.
These results are new, and the approach using Ramsey spaces dense in these forcing  is much more succinct 
than approaching the Ramsey degrees using the forcings 
$\mathcal{P}(\om^k)/\Fin^{\otimes k}$.

In Section \ref{sec:Cardinals}, we investigate pseudointersection and tower numbers related to several classes of topological Ramsey spaces.
   In \cite{Malliaris/ShelahJAMS16}, Malliaris and Shelah solved a longstanding open problem, proving  that 
   $\mathfrak{p}=\mathfrak{t}$.
Any $\sigma$-closed forcing has naturally defined pseudointersection and tower numbers, and it is of interest to know when equality holds between these cardinal invariants.     
Each topological Ramsey space has a natural $\sigma$-closed quasi-ordering such that its separative quotient is isomorphic to the separative quotient of the Ramsey space with its given quasi-order. 
Given such a $\sigma$-closed  order on a topological Ramsey space $\mathcal{R}$, one may define its pseusodintersection number  $\mathfrak{p}_{\mathcal{R}}$ and tower number $\mathfrak{t}_{\mathcal{R}}$
 (see Definition \ref{defn.p_Rt_R}).

In Subsection \ref{subsec.IEP}, we prove  that all topological Ramsey spaces whose members are infinite sequences with a certain amount of independence between the entries of the sequence
 have $\mathfrak{p}_{\mathcal{R}}=\mathfrak{t}_{\mathcal{R}}=\mathfrak{p}$.
Such spaces include those dense in Laflamme's forcings $\mathbb{P}_{\al}$, $1\le\al<\om_1$, the forcings of Baumgartner-Taylor and of Blass, those in \cite{fraisseClasses}.
In contrast, in Subsection \ref{subsec.ptE_al},
we show that pseudointersection and tower numbers associated with the forcings $\mathcal{P}(\om^{\al})/\Fin^{\otimes\al}$ are all $\om_1$.
For $\al$ a finite ordinal, this reproduces results of Kurili\'{c} in \cite{Kurilic15} albeit in what we consider a streamlined fashion.
However, for infinite countable ordinals $\al$, these results are new.
Finally, in Subsection
\ref{subsec.CS} we show that the Carlson-Simpson space forces a Ramsey ultrafilter.
This is especially interesting as it is a dual space, quite different from the other spaces considered in this paper; and while it generates a Ramsey ultrafilter, its pseudointersection and tower numbers are $\om_1$, proved by Matet in
 \cite{matet1986}.

 It remains open whether the pseudointersection and tower numbers are always equal for topological Ramsey spaces,  and whether there are any such spaces for which the pseudointersection or tower numbers can lie strictly between $\om_1$ and $\mathfrak{p}$.
 These and other questions appear in 
Section \ref{sec:questions}.


\section{Classical background and definitions}
\label{sec:Background}

\spnewtheorem*{notation}{Notation}{\bf}{\rm}
\spnewtheorem*{fact}{Fact}{\bf}{\rm}

This section contains basic definitions and background for some classical results on ultrafilters and cardinal invariants on  $\omega$.

\begin{definition}\label{defn.SFIP}
\begin{enumerate}
\item
 For two sets $X, Y \subseteq \omega$ we say that $X$ is \emph{almost contained} in $Y$, denoted $X \subseteq^{\ast} Y$, if $X \setminus Y$ is  finite.
\item
We say that a family of sets  $\mathcal{F} \subseteq [\omega]^{\omega}$ has the \emph{strong finite intersection property} (SFIP) if for every finite subfamily $ \mathcal{X} \in [\mathcal{F}]^{<\om}$, $\bigcap \mathcal{X}$ is an infinite subset of $\omega$.
\item Given $ \mathcal{F} \subseteq [\omega]^{\omega}$, a pseudointersection of the family $\mathcal{F}$ is a set  $Y \in [\omega]^{\omega}$	such that for every $X \in \mathcal{F}$, $ Y \subseteq ^{\ast} X $.
\item The \emph{pseudointersection number} $\mathfrak{p}$ is the smallest cardinality of a family $\mathcal{F} \subseteq [\omega]^{\omega}$ which has the SFIP but does not have a pseudointersection.
\item
A {\em tower} is  a sequence $\lgl X_{\al}:\al<\delta\rgl$ of members of $[\om]^{\om}$
which is linearly ordered by $\supseteq^*$  and has no
 pseudointersection.
The {\em tower number} $\mathfrak{t}$ is the smallest cardinality of  a tower. 
\end{enumerate}
\end{definition}

It is well-known that Martin's Axiom implies $ \mathfrak{p}= \mathfrak{c}$;  see for instance
    \cite{Barto} for a proof.
    The cardinal invariant
$\mathfrak{m}(\sigma$-centered$)$ is  defined to be the
minimum cardinal  $\kappa$
for which
 there exists a $\sigma$-centered partial order
 $\mathcal{P}$ and a family $\mathcal{D}$ of $\kappa$ many  dense subsets of  $\mathcal{P}$ which does not admit any $\mathcal{D}$-generic filter.
 Bell proved that $\mathfrak{m}(\sigma$-centered$)=\mathfrak{p}$ (see \cite{Bell}).
 It is clear from the definitions that $\mathfrak{p}\le\mathfrak{t}$, and   one of the most important longstanding open problems in cardinal invariants  was whether the two are equal.
Malliaris and Shelah recently proved that, indeed, 
 $\mathfrak{p}=\mathfrak{t}$
(see   \cite{pt} and \cite{Malliaris/ShelahJAMS16}).

\begin{definition}
	Let $\mathcal{U}$ be an ultrafilter on  $\om$.
	\begin{enumerate}
		\item
		$\mathcal{U}$ is \emph{Ramsey} if for each coloring $c: [ \omega ]^{2} \rightarrow 2 $, there is a $U \in \mathcal{U}$ such that $U$ is homogeneous for $c$, meaning $| c[U]^{2} | =1$.
		\item
		$\mathcal{U}$ is \emph{weakly Ramsey} if for each coloring $c: [ \omega ]^{2} \rightarrow 3 $, there is a $U \in \mathcal{U}$ such that $| c[U]^{2} | \leq 2$.
	\end{enumerate}
\end{definition}

It is well known that an ultrafilter $\mathcal{U}$  on $\om$ is Ramsey if and only if it is 
 \emph{selective}: 
 For each
		$\supseteq$-decreasing sequence $\langle U_{n} \rangle_{n \in \omega}$ of members of $\mathcal{U}$, there is an $X \in \mathcal{U}$ such that for each $n < \omega$, $X \subseteq^{\ast}U_{n}$ and moreover $|X \cap (U_{n}\backslash U_{n+1})| \leq 1$.

If  $\mathcal{U}$ is a Ramsey ultrafilter, then it is routine to show that,  in the notation of Definition \ref{defn.Ramseydegree}, 
$t(\mathcal{U},n)=1$ for all $n\ge 1$.
An ultrafilter $\mathcal{U}$ is weakly Ramsey if and only if $t(\mathcal{U},2)\le 2$.
It is clear that any Ramsey ultrafilter is weakly Ramsey.  
On the other hand,
Blass proved in \cite{Blass74} that there are  weakly Ramsey ultrafilters which are not Ramsey.
Proofs of the following facts can be found in \cite{Barto}.

\begin{theorem}
$\langle [\omega]^{\omega}, \subseteq^{\ast} \rangle$ forces a Ramsey ultrafilter.
\end{theorem}

Next, we introduce the Ellentuck topology on $[\om]^{\om}$.
In this section, as in the rest of the paper, we use notation and terminology from \cite{Stevolibro}.
Given $a \in [\omega]^{< \omega}$ and $A \in [\omega]^{\omega}$, define $$[a,A]=\{B \in [\omega]^{\omega }: a \sqsubset B \subseteq  A \},$$
where $a\sqsubset B$  denotes that $a$ is an initial segment of $B$.
The {\em Ellentuck topology} on $[\omega]^{\omega}$
is the topology generated by  basic open sets the sets of the form $[a,A]$.
This topology refines the standard metric topology on the Baire space.
The space $[\omega]^{\omega}$ with the Ellentuck topology is called the {\em Ellentuck space}.

\begin{definition}
	A set $\mathcal{X} \subseteq [\omega]^{\omega}$ is called \emph{Ramsey} if for every non-empty basic open set $[a,A]$ there is $B \in [a,A]$ with $[a,B] \subseteq \mathcal{X}$ or $[a,B] \cap \mathcal{X} = \emptyset$.
\end{definition}

While the Axiom of Choice implies that there are subsets of $[\om]^{\om}$ which are not Ramsey (see \cite{Erdos/Rado52}),  restricting to definable sets allowed for the development  of Ramsey theory of subsets of $[\om]^{\om}$.
A highlight of this development was the  Galvin-Prikry Theorem  \cite{Galvin/Prikry73}  that all Borel sets are Ramsey.
 Silver later proved in  \cite{Silver70}  that all analytic sets are Ramsey.
The pinnacle of this line of work was Ellentuck's   topological characterization of those subsets of $[\om]^{\om}$ which are Ramsey.
His work was strongly informed by the above mentioned theorems  as well as closely related work of Mathias, which was published  in \cite{Mathias77}.

\begin{theorem}[Ellentuck, \cite{Ellentuck74}]
	Let $\mathcal{X} \subseteq [\omega]^{\omega}$. 
	Then $\mathcal{X}$ is Ramsey if and only if
	$\mathcal{X}$ has the Baire Property in the Ellentuck topology.
\end{theorem}

After this work of Ellentuck, many similar  topological spaces, 
 members  of which  are   infinite sequences,
 were formed and their Ramsey properties were investigated.
 Their similarities were  noticed by Carlson in \cite{Carlson88}, where he coined the term {\em Ramsey space} and proved a general theorem from which these other results  followed as corollaries.
 This line of work was expanded by Carlson and Simpson in
 \cite{Carlson/Simpson90}.
 In his book \cite{Stevolibro},
 Todorcevic set forth
 a simplified set of axioms which guarantee that a space has  Ramsey  properties similar to the Ellentuck space.
 This is the subject of the next section.


\section{Background on topological Ramsey spaces and their associated ultrafilters}\label{sec:tRs}

In this section, we introduce topological Ramsey spaces and the main theorems which provide advantageous techniques.
 The following definition is taken from \cite{Stevolibro}. 
 The axioms $\mathbf{A.1}$--$\mathbf{A.4}$ are defined for triples $(\mathcal{R}, \leq,r)$ of objects with the following properties, 
 where $\mathcal{R}$ is a nonempty set, $\leq$ is a quasi-ordering on $\mathcal{R}$, and $r: \mathcal{R} \times \omega \longrightarrow \mathcal{AR}$ is a mapping giving us the sequence $(r_{n}(\cdotp) = r( \cdotp ,n))$ of approximation mappings, where 
$$
\mathcal{AR}=\{r_n(A):A\in\mathcal{R}\mathrm{\ and\ }n<\om\}
$$
 is the collection of all {\em finite approximations} to members of $\mathcal{R}$. 
For $a \in \mathcal{AR}$ and $A,B \in \mathcal{R}$, $$[a,B] = \{ A \in \mathcal{R} : A \leq B \, \mathrm{and} \, \exists n \in \omega (r_{n}(A)=a) \}.$$

For $a \in \mathcal{AR}$, let $|a|$ denote  the integer $k$ such that $a=r_{k}(A),$ for some $A \in \mathcal{R}$.
 If $m< n$,  $A\in \mathcal{R}$, $a=r_{m}(A)$ and $b=r_{n}(A)$, then we will write $a=r_{m}(b)$.
For $a,b \in \mathcal{AR}$, $a \sqsubseteq b$ if and only if $a = r_{m}(b)$ for some $m \leq |b|$; if
 $a \sqsubseteq b$ and
 $a\ne b$, we write $a \sqsubset b.$
For each $n < \omega$, $\mathcal{AR}_{n} = \{ r_{n}(A): A \in \mathcal{R} \}$;
then
$\mathcal{AR}=\bigcup_{n<\om}\mathcal{AR}_n$.

\begin{description}
	\item[$\mathbf{A.1}$]
	\begin{itemize}
		\item[(a)] $r_{0}(A)= \emptyset$ for all $A \in \mathcal{R}$.
		\item[(b)] $A \neq B$ implies $r_{n}(A) \neq r_{n}(B)$ for some $n$.
		\item[(c)] $r_{n}(A) = r_{m}(B)$ implies $n=m$ and $r_{k}(A)=r_{k}(B)$ for all $k < n.$
		
	\end{itemize}
	
	\hspace{.5cm}
	
	\item[$\mathbf{A.2}$] There is a quasi-ordering $\leq_{\fin}$ on $\mathcal{AR}$ such that
	\begin{itemize}
		\item[(a)] $\{ a \in \mathcal{AR}: a \leq_{\fin}b \}$ is finite for all $b \in \mathcal{AR}$,
		\item[(b)] $A \leq B$ if and only if for each $n \in \omega$ there exists $m \in \omega$ such that $r_{n}(A) \leq_{\fin} r_{m}(B)$,
		\item[(c)] For every $a,b,c \in \mathcal{AR}$, if $a \sqsubset b$ and $b \leq_{\fin} c$ then there exists $d \in \mathcal{AR}$ such that $d \sqsubset c$ and $a \leq_{\fin}d$.
		\end{itemize}
		
	\hspace{.5cm}
	
The notation 	$\depth_{B}(a)$ is defined to be the least $n$, if it exists, such that $a \leq_{\fin}r_{n}(B)$.
 If such an $n$ does not exists, then write $\depth_{B}(a)= \infty$. If $\depth_{B}(a) =n < \infty$, then $[\depth_{B}(a),B]$ denotes $[r_{n}(B),B]$.
	
	\hspace{.5cm}
	
	\item[$\mathbf{A.3}$]
	\begin{itemize}
		\item[(a)] If $\depth_{B}(a)< \infty$ then $[a,A] \neq \emptyset$ for all $A \in [\depth_{B}(a),B]$.
		\item[(b)] $A \leq B$ and $[a,A] \neq \emptyset$ imply that there is $A' \in [\depth_{B}(a),B]$ such that $\emptyset \neq [a,A'] \subseteq [a,A]$.
	\end{itemize}
	
	\hspace{.5cm}
	
	If $n > |a|$, then $r_{n}[a,A]$ is the collection of all $b \in \mathcal{AR}_{n}$ such that $a\sqsubset b$ and $b= r_m(B)$ for some $B\in [a,A]$.
	
	\hspace{.5cm}
	
	\item[$\mathbf{A.4}$] If $\depth_{B}(a) < \infty$ and if $\mathcal{O} \subseteq \mathcal{AR}_{|a|+1}$, then there is $A \in [\depth_{B}(a),B]$ such that $r_{|a|+1}[a,A] \subseteq \mathcal{O}$ or $r_{|a|+1}[a,A] \subseteq \mathcal{O}^{c}$.

\end{description}

The basic open sets are  exactly those
 sets of the form $[a,B]$, where $a\in\mathcal{AR}$ and $B\in\mathcal{R}$.
The 
\emph{Ellentuck topology} on $\mathcal{R}$ is the 
topology on $\mathcal{R}$ generated by the basic open sets.
The {\em metrizable topology} on 
$\mathcal{R}$ is generated by the sets of the form $\{A\in\mathcal{R}:a\sqsubset A\}$, where $a\in\mathcal{AR}$;
this is the  topology obtained by considering 
$\mathcal{R}$ as a subspace of the Tychonoff cube $\mathcal{AR}^{\omega}$, where $\mathcal{AR}$ is endowed with the discrete topology. 
 Note that the Ellentuck topology is finer than the metrizable topology on $\mathcal{R}$.

 Given the Ellentuck topology on $\mathcal{R}$, the notions of nowhere dense, and hence, of meager are defined in the natural way. Thus, we may say that a subset $\mathcal{X}$ of $\mathcal{R}$ has the property of Baire if $\mathcal{X}= \mathcal{O} \cap \mathcal{M}$ for some Ellentuck open set $\mathcal{O} \subseteq \mathcal{R}$ and Ellentuck meager set $\mathcal{M} \subseteq \mathcal{R}$.

\begin{definition}
	A subset $\mathcal{X}$ of $\mathcal{R}$ is \emph{Ramsey} if for every $\emptyset \neq [a,A]$, there is a $B \in [a,A] $ such that $[a,B] \subseteq \mathcal{X}$ or $[a,B] \cap \mathcal{X} = \emptyset$. $\mathcal{X} \subseteq \mathcal{R}$ is {\em Ramsey null} if for every $\emptyset \neq [a,A]$, there is a $B \in [a,A] $ such that $[a,B] \cap \mathcal{X} = \emptyset$.
\end{definition}

\begin{definition}
	A triple $(\mathcal{R}, \leq, r)$ is a \emph{topological Ramsey} space if every subset of $\mathcal{R}$ with the  property of Baire is Ramsey and every meager subset of $\mathcal{R}$ is Ramsey null.
\end{definition}

The following result is Theorem 5.4 in \cite{Stevolibro}.

\begin{theorem}[Abstract Ellentuck Theorem]
	If $(\mathcal{R}, \leq, r)$ is closed (as a subspace of $\mathcal{AR}^{\omega}$) and satisfies axioms $\mathbf{A.1}$, $\mathbf{A.2}$, $\mathbf{A.3}$ and $\mathbf{A.4}$, then every subset of $\mathcal{R}$ with the  property  of Baire  is Ramsey, and every meager subset is Ramsey null; in other words, the triple $(\mathcal{R}, \leq ,r)$ forms a topological Ramsey space.
\end{theorem}

Given $\mathcal{F} \subseteq \mathcal{AR}$ and $X \in \mathcal{R}$,  let $\mathcal{F}\upharpoonright X$ denote the set of all $s \in \mathcal{F}$ such that  $s=r_{n}(Y)$ for some $n \in \omega$ and $ Y\leq X$.

\begin{definition}
	A family $\mathcal{F} \subseteq \mathcal{AR}$ of finite approximations is
	\begin{enumerate}
		\item \emph{Nash-Williams} if for all $a,b \in \mathcal{F}$, $a \sqsubseteq b$ implies $a=b$.
		\item \emph{Ramsey} if for every partition $\mathcal{F}= \mathcal{F}_{0} \cup \mathcal{F}_{1}$ and every $X \in \mathcal{R}$, there are $Y \leq X$ and $i \in \{0,1\}$ such that $\mathcal{F}_{i}\upharpoonright Y = \emptyset$.
	\end{enumerate}
\end{definition}

The next theorem appears as Theorem 5.17 in \cite{Stevolibro}, and follows from the Abstract Ellentuck Theorem.

\begin{theorem}[Abstract Nash-Williams Theorem]\label{NashWilliams}
	Suppose that $(\mathcal{R}, \leq, r)$ is closed (as a subspace of $\mathcal{AR}^{\omega}$) and   satisfies $\mathbf{A.1}$--$\mathbf{A.4}$. Then every Nash-Williams family of finite approximations is Ramsey.
\end{theorem}

	The  original Nash-Williams Theorem on $\om$ states that each  Nash-Williams subset $\mathcal{F}\sse[\om]^{<\om}$ is Ramsey \cite{NashWilliams65}.

In \cite{Goyo} Mijares introduced a generalization of the quasi-order  $\sse^*$  on $[\omega]^{\omega}$  to topological Ramsey spaces.

\begin{definition}\label{defn.Mijares*}
	For $X, Y \in \mathcal{R}$, write $X \leq ^{\ast} Y$ if there exists $a \in \mathcal{AR} \upharpoonright X $ such that $ [ a,X ] \subseteq [ a,Y ]$. In this case we say that $X$ is an $\emph{almost reduction} $ of $Y$.	
\end{definition}
Note that for each $a \in \mathcal{AR} \upharpoonright X $, there exists $Z \in \mathcal{R}$ such that $a \sqsubseteq Z$ and $Z \leq X$, so $ \emptyset \neq [ a,X ] \subseteq [ a,Y ]$.

\begin{remark}
For   the topological Ramsey spaces considered in this paper, $(\mathcal{R},\leq^{\ast})$ is a $\sigma$-closed partial order such that 
 $(\mathcal{R},\le)$  and  $(\mathcal{R},\leq^{\ast})$  have isomorphic separative quotients.
 Thus,   the two quasi-orders  are  interchangeable from the viewpoint of forcing. 
 However, in certain instances,  even coarser $\sigma$-closed quasi-orders will be used.
 Even so, these will still have separative quotients which are isomorphic to those of $(\mathcal{R},\le)$  and  $(\mathcal{R},\leq^{\ast})$. 
\end{remark}


\section{Several classes of  topological Ramsey spaces and their associated  ultrafilters}\label{sec:tRs&uf}

Given a topological Ramsey space $(\mathcal{R},\le,r)$,  the generic filter forced by $(\mathcal{R},\le)$ induces an ultrafilter as we now show:
In all known examples of topological Ramsey spaces,  the collection of first approximations, $\mathcal{AR}_1$, is a  countable set.
If that is not the case for  some particular space  $\mathcal{R}$, the 
restriction
$\mathcal{AR}_1\re A$ for any 
 member $A$ of $\mathcal{R}$ is  countable  by Axiom \bf A.2\rm, so one may work below a fixed member of $\mathcal{R}$, if necessary.
 
 \begin{definition}\label{defn.U_R}
Given   a generic filter $G\sse \mathcal{R}$  for the forcing $(\mathcal{R},\le)$,
define
\begin{equation}\label{eq.U_R}
\mathcal{U}_{\mathcal{R}}
=\{
S\sse\mathcal{AR}_1:   
S\supseteq \mathcal{AR}_1\re A\mathrm{\ for\ some\ }
A\in G\}.
\end{equation}
\end{definition}

\begin{lemma}\label{Rdegsame}
Let  $(\mathcal{R},\le,r)$ be a topological Ramsey space,
$\le^*$  be a $\sigma$-closed quasi-order  coarsening $\le$, and $G\sse\mathcal{R}$ be a generic filter 
for  $(\mathcal{R},\le^*)$.
Let  $\mathcal{U}_{\mathcal{R}}$  be the filter on base set $\mathcal{AR}_1$ defined  in (\ref{eq.U_R}).
Then $\mathcal{U}_{\mathcal{R}}$  is an ultrafilter on the base set $\mathcal{AR}_1$.
\end{lemma}

\begin{proof}
This  follows from the Abstract Nash-Williams Theorem and  genericity of $G$.
\end{proof}

This section introduces some topological Ramsey spaces and their associated ultrafilters whose Ramsey degrees, pseudointersection and tower numbers will be  investigated in subsequent sections.

\subsection{The topological Ramsey spaces $\mathcal{R}_{\al}$, $1\le \al<\om_1$}\label{subsec4.1}

In \cite{Laflamme}, Laflamme constructed a  forcing, denoted $\mathbb{P}_1$, which generates a weakly Ramsey ultrafilter, denoted $\mathcal{U}_1$, which is not Ramsey.
Although Blass had already shown such ultrafilters exist (see \cite{Blass74}),
the point of  $\mathbb{P}_1$ was to construct a weakly Ramsey ultrafilter with complete combinatorics,
 analagous to
 the result that 
  any Ramsey ultrafilter in the model $V[G]$ obtained by L\'{e}vy collapsing a Mahlo cardinal to $\aleph_1$ is $([\om]^{\om},\sse^*)$-generic over HOD$(\mathbb{R})^{V[G]}$
  (see \cite{Blass88} and 
   \cite{Mathias77}).
 One of the advantages of forcing  with topological Ramsey spaces  is that the associated ultrafilter automatically has complete combinatorics in the presence of large cardinals (see \cite{DipriscoMijares} for the result and \cite{Dobrinen/Seals} for an overview of this area).
 In  \cite{R1},  a topological Ramsey space denoted
 $\mathcal{R}_{1}$  
  was constructed 
  which forms a dense subset of Laflamme's forcing $\mathbb{P}_1$, hence generating the same weakly Ramsey ultrafilter.
  The motivation for that construction was to 
  find the exact  Tukey structure below $\mathcal{U}_1$ as well as the precise structure of the Rudin-Keisler classes  within these Tukey types, which were indeed  found in \cite{R1}.
Here, we reproduce a few  definitions and facts relevant to  this paper.

\begin{definition}[$(\mathcal{R}_{1}, \leq, r)$, \cite{R1}].
	Let $\mathbb{T}_1$ denote the following infinite tree of height 2.
	$$ \mathbb{T}_1= \{\langle \rangle\} \cup \{\langle n \rangle : n < \omega  \} \cup \bigcup_{n < \omega}\{\langle n,i \rangle : i \leq n \}.
	$$
	$\mathbb{T}_1$ can be thought of as an infinite sequence of finite trees of height 2, where the $n$-th \emph{subtree} of $\mathbb{T}_1$ is $$\mathbb{T}_1(n)= \{\langle \rangle, \langle n \rangle, \langle n,i \rangle : i \leq n \}.$$

	The members of $\mathcal{R}_{1}$ are infinite subtrees of $\mathbb{T}_1$ which have the same structure as $\mathbb{T}_1$.
	That is, a tree $X \subseteq \mathbb{T}_1$ is in $\mathcal{R}_{1}$ if and only if there is a strictly increasing sequence $(k_{n})_{n < \omega}$ such that
	\begin{enumerate}
		\item 
		$X \cap \mathbb{T}_1(k_{n}) \cong \mathbb{T}_1(n)$ for each $n < \omega$; and
		\item 
		whenever $X \cap \mathbb{T}_1(j) \neq \emptyset$, then $j = k_{n}$ for some $n < \omega$.
	\end{enumerate}
	When this holds, we let $X(n)$ denote $X \cap \mathbb{T}_1(k_{n})$, and call $X(n)$ the \emph{n-th subtree} of $X$. For $n < \omega$, $r_{n}(X)$ denotes $\bigcup_{i < n}X(i)$.

	For $X,Y \in \mathcal{R}_{1}$, define $Y \leq X$ if and only if there is
	 a strictly increasing sequence $(k_{n})_{n < \omega}$ such that for each $n$, $Y(n)$ is a subtree of $X(k_{n})$.
Notice that by the structure of the members of $\mathcal{R}_1$,
$Y\le X$ exactly when $Y\sse X$.
Given  $a,b \in \mathcal{AR}$,
define
	$b \leq_{\fin} a$ if and only if $b\sse a$.
\end{definition}

The following figure presents the first five ``blocks'' of the maximal member  of  $\mathcal{R}_1$.

\begin{figure}[H]
		\centering
		{\footnotesize
			\begin{tikzpicture}[scale=.7,grow'=up, level distance=40pt,sibling distance=.2cm]
			\tikzset{grow'=up}
			\Tree [.$\emptyset$ [.$\langle0\rangle$  [.$\langle 0,0\rangle$ ] ]
[.$\langle1\rangle$ [.$\langle1,0\rangle$ ][.$\langle1,1\rangle$ ] ] [.$\langle2\rangle$ [.$\langle2,0\rangle$ ] [.$\langle2,1\rangle$ ][.$\langle2,2\rangle$ ] ] [.$\langle3\rangle$ [.$\langle3,0\rangle$ ]  [.$\langle3,1\rangle$ ] [.$\langle3,2\rangle$ ] [.$\langle3,3\rangle$ ] ] [.$\langle4\rangle$ [.$\langle4,0\rangle$ ] [.$\langle4,1\rangle$ ] [.$\langle4,2\rangle$ ] [.$\langle4,3\rangle$ ][.$\langle4,4\rangle$ ] ]]
			\end{tikzpicture}}
		\caption{$r_{5}(\mathbb{T}_1)$}	
		\end{figure}
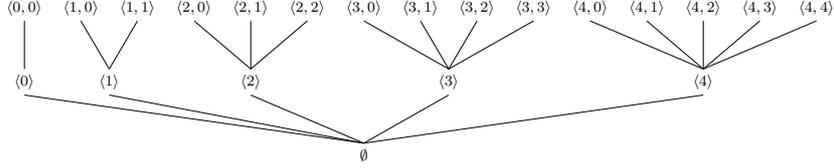

The members of $\mathcal{R}_1$ are subtrees of 
$\mathbb{T}_1$  which are isomorphic to $\mathbb{T}_1$.
As the first step toward the main theorem of
 \cite{R1}, the following was proved.

\begin{theorem}[Dobrinen and Todor\v{c}evi\'{c}, \cite{R1}]
	$(\mathcal{R}_{1}, \leq, r )$ is a topological Ramsey space.	
\end{theorem}

Notice that by the structure of the members of $\mathcal{R}_1$,
given $X,Y\in\mathcal{R}_1$,
$Y\le^* X$ (recall Definition \ref{defn.Mijares*})
holds
 if and only if there is an $i<\om$
and a strictly increasing sequence $(k_n)_{n\ge i}$ such that for each $n\ge i$,  $Y(n)\sse X(k_n)$.
Thus,  the quasi-order $\le^*$ 
turns out to be equivalent to  $\sse^*$, since 
$Y\le^* X$ if and only if $Y\sse^* X$.
By an  {\em ultrafilter $\mathcal{U}_{\mathcal{R}_1}$ associated with the forcing $(\mathcal{R}_1,\le^*)$}  we mean the
ultrafilter   on base set $\mathcal{AR}_1$
generated by 
 the sets $\mathcal{AR}_1\re X$, $X\in G$, where $G$ is some 
  generic filter for  $(\mathcal{R}_1,\le^*)$.
 By the density of this topological Ramsey space in Laflamme's forcing,
this ultrafilter $\mathcal{U}_{\mathcal{R}_1}$ is isomorphic to
the ultrafilter $\mathcal{U}_1$ generic for Laflamme's forcing $\mathbb{P}_1$.
Hence, it is weakly Ramsey but not Ramsey.

Continuing in this vein,
 Laflamme constructed  a hierarchy of forcings $\mathbb{P}_{\al}$, $1\le \al<\om_1$,  in order to produce rapid p-points   $\mathcal{U}_{\al}$ satisfying partition relations with decreasing strength as $\al$ increases, and such that for $\beta <\al$, $\mathcal{U}_{\beta}$ is Rudin-Keisler below $\mathcal{U}_{\al}$.
In \cite{Laflamme},
Laflamme proved  
that  each $\mathcal{U}_{\al}$  has complete combinatorics, and 
that below $\mathcal{U}_{\al}$, there is   a decreasing chain of length $\al+1$  of Rudin-Keisler types, the least one being that of  a Ramsey ultrafilter.
This left open, though, whether or not this chain is the only Rudin-Keisler structure below  $\mathcal{U}_{\al}$.

Topological  Ramsey spaces 
$\mathcal{R}_{\al}$ were  constructed in \cite{Rn} to produce  dense subsets of Laflamme's forcings $\mathbb{P}_{\al}$, hence  generating  the same generic ultrafilters.
The reader is referred to \cite{Rn} for the definition of these spaces. 
The Ramsey space techniques provided valuable methods  for proving  in \cite{Rn} that indeed  the Rudin-Keisler, and moreover, the  Tukey structure below $\mathcal{U}_{\al}$ is exactly a chain of length $\al+1$.
Here, we reproduce $\mathcal{R}_2$, with a minor modification  not affecting its forcing properties  which will make it easier to understand.  
The reader can then infer the structure of $\mathcal{R}_k$ for each $1\le k<\om$.
In Section \ref{sec:degrees}, we will only work with $\mathcal{R}_k$ for $1\le k<\om$, since the Ramsey degree  $t(\mathcal{U}_{\om},2)=\om$. 
However, Section 
\ref{sec:Cardinals}
 will consider  pseudointersection and tower numbers of $\mathcal{R}_{\al}$, for all $1\le \al<\om_1$.


\subsection{Ramsey spaces from \Fraisse\ classes}\label{subsec.Fraisse}

This subsection introduces topological Ramsey spaces  constructed   in \cite{fraisseClasses}.
The motivation for these spaces was to  find dense subsets of some forcings of Blass in \cite{Blass73} and of Baumgartner and Taylor in \cite{Baumgartner/Taylor78} 
in order to better study  properties of their forced  ultrafilters (more details provided below).
The construction was seen to easily generalize to any \Fraisse\ classes with the Ramsey property. 
We provide the  basic ideas of the  construction here.

\begin{definition}[The space $\mathcal{R}(\mathbb{A})$, \cite{fraisseClasses}]
Fix some natural number $J\ge 1$, and for each 
$j< J$, let $\mathcal{K}_j$ be a \Fraisse\ class of finite linearly ordered relational structures with the Ramsey property. 
We say that 
 $\mathbb{A}=\lgl (\mathbf{A}_{k,j})_{k<\om}:j<J\rgl$
 is a {\em generating sequence} if for each $j<J$, the following hold:
 \begin{enumerate}
 \item
For each $k<\om$, $\mathbf{A}_{k,j}$  is a member of $\mathcal{K}_j$,
 and  $\mathbf{A}_{0,j}$ has universe of cardinality $1$.
\item
Each $\mathbf{A}_{k,j}$ is a substructure of 
$\mathbf{A}_{k+1,j}$.
\item
For each structure $\mathbf{B}\in\mathcal{K}_j$, there is a $k$ such that $\mathbf{B}$ embeds into $\mathbf{A}_{k,j}$.
\item
For each pair $k<m<\om$, there is an $n>m$ large enough that the following Ramsey property holds:
$$
\mathbf{A}_{n,j}\rightarrow (\mathbf{A}_{m,j})^{\mathbf{A}_{k,j}}.
$$
\end{enumerate}
 Let $\mathbf{A}_{k}$ denote the
 $n$-tuple of structures $(\mathbf{A}_{k,j})_{j<J}$.
 It can be convenient to think of this as the 
  product $\prod_{j< J} \mathbf{A}_{k,j}$ with no additional relations.
 Let $\mathbb{A}= \langle \langle k, \mathbf{A}_{k} \rangle : k < \omega \rangle $. 
 This infinite sequence  $\mathbb{A}$  of $J$-tuples of finite structures
is the maximal member of 
the space $\mathcal{R}(\mathbb{A})$.
 We define $B$ to be a member of $\mathcal{R}(\mathbb{A})$ if and only if $B = \langle \langle n_{k}, \mathbf{B}_{k} \rangle : k < \omega \rangle$, where
	\begin{enumerate}
		\item $(n_{k})_{k < \omega}$ is some strictly increasing sequence of natural numbers; and
		\item for each $k < \omega$, $\mathbf{B}_{k}$ is  an  $J$-tuple  $(\mathbf{B}_{k,j})_{j<J}$,
		where each 
	$ \mathbf{B}_{k,j}$ is a substructure of 
	$\mathbf{A}_{n_{k},j}$
	isomorphic to $\mathbf{A}_{k,j}$.
	\end{enumerate}
	We use $B(k)$ to denote $\langle n_{k}, \mathbf{B}_{k} \rangle$, the $k$-th block of $B$. The $m$-th approximation of $B$ is $r_{m}(B)=\langle B(0),...,B(m-1)\rangle$.

	Define the partial order $\leq$ as follows:   For $B= \langle \langle m_{k}, \mathbf{B}_{k} \rangle : k < \omega \rangle$ and $C= \langle \langle n_{k}, \mathbf{C}_{k} \rangle : k < \omega \rangle$, define $C \leq B$ if and only if for each $k$ there is an $l_{k}$ such that $n_{k} = m_{l_{k}}$ and for all $j <J$, $ \mathbf{C}_{k,j}$ is a substructure of  $\mathbf{B}_{l_{k},j} $.
The partial order $\leq_{\fin}$ on the collection of finite approximations, $\mathcal{AR}$, is defined  as follows: 
For $b= \langle \langle m_{k}, \mathbf{B}_{k} \rangle : k < p \rangle \rangle $ and $c= \langle \langle n_{k}, \mathbf{C}_{k} \rangle : k < q \rangle$, where $p,q<\om$, define $c \leq_{\fin} b$ if and only if there are $C \leq B$ such that $c= r_{q}(C)$, $b=r_{p}(B)$.
For these spaces, the naturally associated $\sigma$-closed partial order $\le^*$  from Definition \ref{defn.Mijares*} is simply $\sse^*$.
\end{definition}

\begin{theorem}[Dobrinen, Mijares and Trujillo,  \cite{fraisseClasses}]
Given a generating sequence $\lgl (\mathbf{A}_{k,j})_{k<\om}:j<J\rgl$,
the triple  $(\mathcal{R}(\mathbb{A}), \le r)$ forms a topological Ramsey space. 
\end{theorem}

Letting  $\mathcal{R}$ denote $\mathcal{R}(\mathbb{A})$, 
given a generic filter $G$ for the forcing  $(\mathcal{R},\le^*)$,
we  let $\mathcal{U}_{\mathcal{R}}$
denote the ultrafilter on base set $\mathcal{AR}_1$ generated by the sets $\mathcal{AR}_1\re X$, $X\in G$.
The motivation for these spaces
came from studying the Tukey types below ultrafilters 
constructed in   \cite{Blass73}  and \cite{Baumgartner/Taylor78}.
The special case where $n=2$ and both $\mathcal{K}_0$ and $\mathcal{K}_1$ are the classes of finite linear orders produces a Ramsey space which is dense inside the {\em $n$-square forcing} of Blass in \cite{Blass73}, which he constructed to produce a p-point which has  two Rudin-Keisler incomparable selective ultrafilters Rudin-Keisler  below it. 
Given $n\ge 2$, we shall let $\mathcal{H}^n$ denote the Ramsey space produced when each $\mathcal{K}_j$, $j<n$, is the class of finite linear orders;  call this space the {\em $n$-hypercube space}.
The space $\mathcal{H}^2$ is dense in Blass' forcing, and hence  the ultrafilter $\mathcal{U}_{\mathcal{H}^2}$ is isomorphic to the one constructed by Blass.
The collection of spaces $\mathcal{H}^n$, $n\ \ge 2$, form  a hierarchy of forcings such that each ultrafilter $\mathcal{U}_{\mathcal{H}^n}$ projects to the ultrafilter $\mathcal{U}_{\mathcal{H}^m}$ for $m<n$.
It is shown in \cite{fraisseClasses} that the initial Tukey structure below $\mathcal{U}_{\mathcal{H}^n}$  is isomorphic to the Boolean algebra $\mathcal{P}(n)$.
In another direction, the special cases where $J=1$, $k\ge 3$ is fixed,  and $\mathcal{K}_0$ is the class of all finite ordered $k$-clique-free graphs  produces Ramsey spaces which are dense inside partial orders constructed by Baumgartner and Taylor in
\cite{Baumgartner/Taylor78} which produce p-points which have asymmetric partition relations, called $k$-arrow ultrafilters. 
Results on the  initial Rudin-Keisler and  Tukey  structures  of ultrafilters constructed by Ramsey spaces  from generating sequences
 appear in \cite{fraisseClasses}, which includes some work of Trujillo in his 
 thesis \cite{TrujilloThesis}.


	\subsection{High dimensional Ellentuck spaces}\label{subsec.hdE}

The next topological Ramsey spaces we  present are the high dimensional Ellentuck spaces. 
We shall let $\mathcal{E}_{1}$ denote the Ellentuck space; that is $([\om]^{\om},\sse,r)$, where for $X\in[\om]^{\om}$ and  $n<\om$,  $r_n(X)=\{x_i:i<n\}$ where $\{x_i:i<\om\}$ is the increasing enumeration of $X$. 
The first new space, $\mathcal{E}_{2}$, was motivated by a  problem left open in \cite{Blass/Dobrinen}: 
 finding the precise structure  of the ultrafilters Tukey reducible to the generic ultrafilter forced by $\mathcal{P}(\omega^{2})/ \Fin^{\otimes 2}$, denoted by  $\mathcal{G}_{2}$.
 Here, $\Fin^{\otimes 2}$ is the ideal of all subsets  $I\sse \om\times\om$  for which all but finitely
  many fibers are finite; that is, for all but finitely many $n$, the set $\{i<\om: (n,i)\in I\}$ is finite.
 This ultrafilter $\mathcal{G}_2$ is not a p-point, but still satisfies the partition relation $\mathcal{G}_2\ra(\mathcal{G}_2)^2_{l,4}$, which is the best partition relation a non-p-point can possess.
Similarly to the construction of  $\Fin^{\otimes 2}$, 
ideals $\Fin^{\otimes k+1}$ can be recursively defined to consist of those subsets  $X$ of $\om^{k+1}$ such that for all but finitely many $n$, the collection $\{(j_1,\dots,j_k): (n,j_1,\dots,j_k)\in X\}$ is a member of $\Fin^{\otimes k}$.
Each ideal $\Fin^{\otimes k}$ is $\sigma$-closed under the partial order $\sse^{\Fin^{\otimes k}}$, and the Boolean algebras $\mathcal{P}(\omega^{k})/ \Fin^{\otimes k}$, $k\ge 2$,  force   ultrafilters,  denoted  $\mathcal{G}_{k}$, which form a hierarchy in the sense that $\mathcal{G}_j$ is recovered as the projection of $\mathcal{G}_k$ to  its first $j$ coordinates, for  $ j<k$.
We let $\mathcal{G}_1$ denote the ultrafilter forced by $\mathcal{P}(\om)/\Fin$;
this is 
 a Ramsey  ultrafilter.

High dimensional Ellentuck spaces   $\mathcal{E}_k$  were constructed by the first author in order to  form topological Ramsey spaces which are forcing equivalent to the Boolean algberas $\mathcal{P}(\om^k)/\Fin^{\otimes k}$.
This was utilized 
 to 
find   the exact Rudin-Keisler  and Tukey structures below each $\mathcal{G}_k$;
 these turn out to be exactly chains of length $k$ (see   \cite{highDimensional}). 
Having already proved their forcing equivalence 
 to the Boolean algebras $\mathcal{P}(\omega^{k})/ \Fin^{\otimes k}$  in \cite{highDimensional}, 
 we provide here the simpler sequence version of these  spaces:
 These are the ``domain spaces'' in \cite{DobrinenJML16}, 
  and this formulation can be found also in 
 \cite{highDimensional2}.

\begin{definition} \label{defOmle}
For $k \ge 2$, denote by $\omle{k}$ the collection of  all non-decreasing sequences of members
of $\omega$ of length less than or equal to $k$.
\end{definition}

The lexicographic order on $\omle{k}$ is defined as usual; it is presented here to aid the reader. 

\begin{definition}[The lexicographic order on $\omle{k}$] \label{defLex}
Let $(s_0,\dots,s_i),(t_0,\dots,t_j)\in\omle{k}$ be given.
We say that
$(s_0,\dots,s_i)$ is {\em lexicographically below} $(t_0,\dots,t_j)$, written $(s_0,\dots,s_i) \lex
(t_0,\dots,t_j)$, if and only if there is a non-negative integer $m$ with the following properties:
\begin{enumerate}
\item[(i)]   $m\le\min(i, j)$;
\item[(ii)]  for each  $n <m, s_n = t_n$; and
\item[(iii)] either $s_{m} < t_{m}$, or else $m = i< j$ and $s_i=t_i$.
\end{enumerate}
\end{definition}

Whereas for $k\ge 2$, $<_{\mathrm{lex}}$ has order type greater than $\om$, the next well-order has order-type $\om$ for any $k\ge 2$.
This  will aid in defining the finite approximations of members of $\mathcal{E}_k$, and is central to seeing how each $\mathcal{E}_k$ can be obtained as a projection of 
$\mathcal{E}_{k+1}$.

\begin{definition}[The well-ordered set $(\omle{k},\prec)$] \label{defPrec}
Set the empty sequence $()$ to be the $\prec$-minimum element of $\omle{k}$; so, for all nonempty
sequences $s$ in $\omle{k}$, we have $() \prec s$. In general, given  $(s_0,\dots,s_i)$ and
$(t_0,\dots,t_j)$ in $\omle{k}$ with $i,j \ge 1$, define $(s_0,\dots,s_i) \prec (t_0,\dots,t_j)$
if and only if either
\begin{enumerate}
\item
$s_i < t_j$, or
\item
$s_i = t_j$ and
$(s_0,\dots,s_i) <_{\mathrm{lex}} (t_0,\dots,t_j)$.
\end{enumerate}
\end{definition}

\noindent \textbf{Notation.}
Let $\ome{k}$ denote the collection of all non-decreasing sequences of length $k$ of members of
$\omega$, and 
notice  that the $\prec$ also   well-orders $\ome{k}$ in order type $\omega$.
 Let 
$\vec{u}_n$ denote the $n$-th member of $(\ome{k},\prec)$.
For $s,t\in \omle{k}$,
we say that $s$ is  a {\em  proper initial segment} of $t$ and write $s\subset t$ if
$s=(s_0,\dots,s_i)$, $t=(t_0,\dots,t_j)$,
$i<j$,
and for all $m\le i$, $s_m=t_m$.

\begin{definition}[The spaces $(\mathcal{E}_k,\le,r)$, $2\le k<\om$, \cite{highDimensional}] \label{defE_k}
An {\em $\mathcal{E}_k$-tree}
 is a function   $\widehat{X}$ from $\omle{k}$ into $\omle{k}$ that preserves the well-order $\prec$  and proper  initial segments $\subset$.
For $\widehat{X}$ an $\mathcal{E}_k$-tree, let $X$ denote the restriction of $\widehat{X}$ to
$\ome{k}$.
The space $\mathcal{E}_k$  is defined to be the collection of all $X$, where   $\widehat{X}$
is an $\mathcal{E}_k$-tree.
We identify $X$ with its range, which is a subset of $\ome{k}$,
 and usually will write $X = \set{x_0,x_1,\ldots}$, where
$x_0 = X(\vec{u}_0) \prec x_1 = X(\vec{u}_1) \prec \cdots$.
The partial ordering on $\mathcal{E}_k$ is defined to be simply inclusion; that is, given $X,Y\in \mathcal{E}_k$,
$X\le Y$ if and only if  (the range of) $X$ is a subset of  (the range of) $Y$.
For each $n<\omega$, the $n$-th restriction function $r_n$ on $\mathcal{E}_k$  is defined
 by
 $r_n(X) = \set{x_i:i<n}$
that is,
the  $\prec$-least $n$  members of
$X$.
We set
\begin{equation}
\mathcal{AE}_n^k := \{ r_n(X) : X \in \mathcal{E}_k \}
\hspace{1.1cm} \textnormal{and} \hspace{1.1cm}
\mathcal{AE}^k := \{ r_n(X) : n < \omega, X \in \mathcal{E}_k \}
\end{equation}
to denote the set of all
$n$-th approximations to members of $\mathcal{E}_k$,
and the set of all
finite approximations to members of $\mathcal{E}_k$, respectively.
\end{definition}


\begin{theorem}[Dobrinen, \cite{highDimensional}]
	For each $2 \leq k < \omega$, $(\mathcal{E}_{k}, \leq, r)$ is a topological Ramsey space.
	Moreover, 
	$(\mathcal{E}_{k},\subseteq^{\Fin^{\otimes k}})$ is forcing equivalent to $\mathcal{P}(\omega^{k}) /\Fin^{\otimes k}$.
\end{theorem}

In  the notation of this paper,  
given a generic filter $G$ for $(\mathcal{E}_{k},\subseteq^{\Fin^{\otimes k}})$, let 
$\mathcal{U}_{\mathcal{E}_k}$  denote the ultrafilter on base set $\mathcal{AE}^k_1$ generated by the sets $\mathcal{AE}_1^k\re X$, $X\in G$. 
Since $(\mathcal{E}_{k},\subseteq^{\Fin^{\otimes k}})$ and $\mathcal{P}(\omega^{k}) / \Fin^{\otimes k}$ are forcing equivalent, 
$\mathcal{U}_{\mathcal{E}_k}$  is isomorphic to 
$\mathcal{G}_{k}$, defined above.

We now provide some details for the spaces $\mathcal{E}_2$ and $\mathcal{E}_3$ in order to provide the reader with more intuition.

\begin{example}(The space $\mathcal{E}_{2}$)
	The members of $\mathcal{E}_{2}$ look like $\omega$ many copies of the Ellentuck space,  with  finite approximations obeying the   well-ordering $\prec$. 
	The well-order $\langle \omle{2}, \prec  \rangle$	begins as follows: $$() \prec (0) \prec (0,0) \prec (0,1) \prec (1) \prec (1,1) \prec (0,2) \prec (1,2) \prec (2) \prec (2,2) \prec \ldots$$
	The  lexicographic order on   $ \omle{2} $   has order type equal to 
	the countable ordinal $\omega^{2}$.
	Here, we picture an initial segment of $\omle{2}$; the set of  maximal nodes is precisely $r_{15}(\ome{2})$.

	\begin{figure}[H]
		\centering
		{\footnotesize
			\begin{tikzpicture}[scale=.7,grow'=up, level distance=40pt,sibling distance=.2cm]
			\tikzset{grow'=up}
			\Tree [.$()$ [.$(0)$ [.$(0,0)$ ][.$(0,1)$ ][.$(0,2)$ ][.$(0,3)$ ] [.$(0,4)$ ] ] [.$(1)$ [.$(1,1)$ ][.$(1,2)$ ][.$(1,3)$ ][.$(1,4)$ ] ] [.$(2)$ [.$(2,2)$ ] [.$(2,3)$ ][.$(2,4)$ ] ] [.$(3)$ [.$(3,3)$ ] [.$(3,4)$ ] ] [.$(4)$ [.$(4,4)$ ] ] ]
			\end{tikzpicture}
		}
		\caption{The initial structure of $\omle{2}$}
	\end{figure}
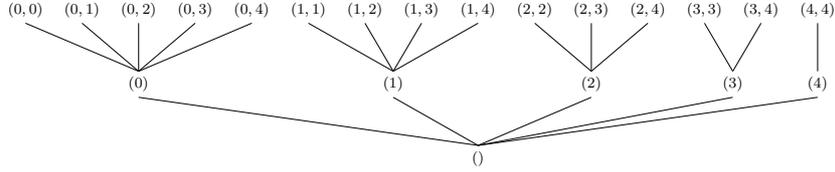
\end{example}

\begin{example}(The space $\mathcal{E}_{3}$)
	The well-order	$\langle \omle{3}, \prec \rangle$ begins as follows:
	\begin{align*}
	\emptyset & \prec (0) \prec (0,0) \prec (0,0,0) \prec (0,0,1) \prec (0,1)  \prec (0,1,1) \prec (1) \\ & \prec (1,1) \prec (1,1,1) \prec (0,0,2) \prec (0,1,2) \prec (0,2) \prec (0,2,2)  \\ & \prec (1,1,2) \prec (1,2) \prec (1,2,2) \prec (2) \prec (2,2) \prec (2,2,2) \prec (0,0,3) \prec \ldots
	\end{align*}
	
	The set $\omle{3}$ is a tree of height three with each non-maximal node branching into $\omega$ many nodes. The following figure shows the initial structure of $\omle{3}$; the set of maximal nodes forms $r_{20}(\ome{3})$.
	
	\begin{figure}[H]
		\centering
		{\footnotesize
			\begin{tikzpicture}[scale=.5,grow'=up, level distance=50pt,sibling distance=.2cm]
			\tikzset{grow'=up}
			\Tree [.$()$ [.$(0)$ [.$(0,0)$ [.$(0,0,0)$ ][.$(0,0,1)$ ][.$(0,0,2)$ ][.$(0,0,3)$ ] ][.$(0,1)$ [.$(0,1,1)$ ][.$(0,1,2)$ ][.$(0,1,3)$ ] ][.$(0,2)$ [.$(0,2,2)$ ][.$(0,2,3)$ ] ][.$(0,3)$ [.$(0,3,3)$ ] ] ] [.$(1)$ [.$(1,1)$ [.$(1,1,1)$ ][.$(1,1,2)$ ][.$(1,1,3)$ ] ][.$(1,2)$  [.$(1,2,2)$ ][.$(1,2,3)$ ]][.$(1,3)$  [.$(1,3,3)$ ] ] ] [.$(2)$ [.$(2,2)$  [.$(2,2,2)$ ] [.$(2,2,3)$ ]] [.$(2,3)$  [.$(2,3,3)$ ] ] ] [.$(3)$ [.$(3,3)$ [.$(3,3,3)$ ] ] ] ]
			\end{tikzpicture}
		}
		\caption{$\omle{3}$}
	\end{figure}
\end{example}


\subsection{The spaces $\FIN_{k}^{[\infty]}$}\label{subsec.FINk}

Next, we introduce a collection of  topological Ramsey spaces that contain infinite sequences of functions.  
The space  $\FIN_1^{[\infty]}$, also denoted simply as $\FIN^{[\infty]}$,
is connected with the famous 
Hindman's Theorem  \cite{Hindman74}.  
Milliken later proved that it forms a topological Ramsey space \cite{Milliken81}. 
The general spaces for $k\ge 2$   are based on work of  Gowers  in \cite{Gowers}.
The presentation here comes from \cite{Stevolibro}.

\begin{definition}
	For a positive integer $k$, define
	 $$
	  \FIN _{k}=\{f : \mathbb{N} \longrightarrow \{0,1,...,k \}:\{n:f(n) \neq 0 \} \text{ is finite and } k \in \mbox{range}(f)\}.
	  $$
	We consider $\FIN _{k}$ a partial semigroup under the operation	of taking the sum of two disjointly supported elements of $\FIN _{k}$. 
	For $f \in \FIN_{k}$, let $\supp(f)=\{ n: f(n) \neq 0 \}$. 
	A block sequence of members of $\FIN _{k}$ is a (finite or infinite) sequence $F=(f_{n})$ such that $$\max \supp( f_{m} ) < \min \supp( f_{n} ) \text{ whenever } m<n.$$

	For $1 \leq d \leq \infty$, let $\FIN^{[d]}_{k}$ denote the collection of all block sequences of length
	$d$.
	The notion of a partial subsemigroup generated by a given block sequence depends on the operation $T: \FIN _{k} \longrightarrow \FIN_{k-1}$ defined as follows: 
	$$
	T(f)(n)= \max \{f(n)-1,0 \}.
	$$
	Given a finite or infinite block sequence $F=(f_{n})$ of elements of $\FIN_{k}$ and an integer $j$ $(1 \leq j \leq k)$, 
	the partial subsemigroup $[F]_{j}$ of $\FIN_{j}$ generated by $F$ is the collection of members of $\FIN _{j}$ of the form $$ T^{(i_{0})}(f_{n_{0}})+...+ T^{(i_{l})}(f_{n_{l}})$$ for some finite sequence $n_{0}<...<n_{l}$ of nonnegative integers and some choice $i_{0},...,i_{l} \in \{0,1,...,k \}$. For $F=(f_{n})$, $G=(g_{n}) \in \FIN _{k}^{[ \leq \infty]}$, set $F \leq G$ if $f_{n} \in [G]_{k}$ for all $n$ less than the length of the sequence $F$. 
	Whenever $F \leq G$, we say that $F$ is a {\em block-subsequence} of $G$. 
	The partial ordering $\le$ on $\FIN _{k}^{[\infty]}$ allows the finitization $\leq_{\fin}$: 
	For $F,G\in \FIN _{k}^{[ < \infty]}$,
	$$F \leq_{\fin}G \text{ if and only if } F \leq G \emph{ and }(\forall l < \text{length}(G)) F\nleq G\upharpoonright l.$$
\end{definition}

\begin{theorem}[\cite{Stevolibro}]
	For every positive integer $k$, the triple $(\FIN _{k}^{[\infty]}, \leq, r)$ is a topological Ramsey space.
\end{theorem}

The  space $\FIN^{[\infty]}_1$, also denoted simply as  $\FIN^{[\infty]}$, was  proved  to be a 
topological Ramsey space by Milliken in \cite{Milliken81}.
 This was the first 
 space that was built on the basis of a substantially different pigeon hole principle, according to Todorcevic in \cite{Stevolibro}. 
 Its power over the Ellentuck space was not fully realized until Gower's succesful applications of the ``Block Ramsey theory" when treating some problems from Banach space geometry.

 The ultrafilter $\mathcal{U}_{\FIN^{[\infty]}}$ associated with the space $\FIN^{[\infty]}$
 is  exactly a  stable ordered-union ultrafilter, in the terminology of \cite{Blass87}.
 Given $f\in \FIN$, let $\min(f)$ denote the minimum of the support of $f$, and let $\max(f)$ denote the maximum of the support of $f$.
 In \cite{Blass87}, Blass showed that the $\min$ and $\max$ projections of the 
ultrafilter  $\mathcal{U}_{\FIN^{[\infty]}}$
 are selective ultrafilters which are Rudin-Keisler incomparable.
 In \cite{Dobrinen/Todorcevic11}, 
 the analogous result for the Tukey order was shown.
More recently, 
Mildenberger 
 showed  in   \cite{Mildenberger11} that  forcing with  $\lgl \FIN^{[\infty]}_k,\le^*\rgl$  produces an ultrafilter  with at least $k+1$-near coherence classes of ultrafilters Rudin-Keisler below it.

\subsection{The Carlson Simpson dual Ramsey space}\label{subsec.CS}

Infinite dimensional dual Ramsey theory was developed by  Carlson and  Simpson in \cite{CS},
where they 
 establish a combinatorial theorem which is the dual of Ellentuck's Theorem.
  The dual form is concerned with colorings of the $k$-element partitions of a fixed infinite set.
 Now, we will introduce the Carlson-Simpson space, also called the  dual Ramsey space.

Using notation in \cite{Stevolibro}, 
let $\mathcal{E}_{\infty}$ denote the collection of all equivalence relations $E$ on $\omega$ with infinity many equivalence classes. 
Each equivalence class $[x]_{E}$ of $E$ has a minimal representative. 
Let $p(E)$ denote  the set of all minimal representatives of classes of $E$, and
let $\{p_{n}(E):n<\om\}$ be the increasing enumeration of $p(E)$. 
Note that  for each $E \in \mathcal{E}_{\infty}$,  $0 \in p(E)$  and hence
 $p_{0}(E)=0$.

For $E,F \in \mathcal{E}_{\infty}$ we say that $E$ is \emph{coarser} than $F$ and write $E \leq F$ if every  equivalence class of $E$  is  the union of   some finitely or infinitely many equivalence classes of $F$. 
The $n$-th \emph{approximation} of  $E \in \mathcal{E}_{\infty}$ is defined as follows:
$$r_{n}(E) = E\upharpoonright p_{n}(E).$$
Thus, $r_{n}(E)$ is simply the restriction of the equivalence relation $E$ to the finite set $\{0,1,...,p_{n}(E)-1\}$ of integers. 
Let 
$$
\mathcal{AE}_{\infty}
=\{r_n(E):E\in\mathcal{E}_{\infty}\mathrm{\ and\ }n<\om\}.
$$
Given $a \in \mathcal{AE}_{\infty}$, the 
 {\em length}  of $a$, denoted $|a|$,  is the integer $n$ such that $a=r_{n}(E)$ for some $E \in \mathcal{E}_{\infty}$.
  (Equivalently, $|a|$ is the number of equivalence classes of $a$.)
  The {\em domain} of $a$ is 
the integer $p_{|a|}(E)= \{0,1,...,p_{|a|}(E)-1 \}$, 
where $E$ is some member of $\mathcal{E}_{\infty}$ such that $a = r_{|a|}(E)$.

\begin{theorem}[Carlson and Simpson, \cite{CS}]
	The space $(\mathcal{E}_{\infty}, \leq,r)$ is a topological Ramsey space.
\end{theorem}

The space  $\mathcal{E}_{\infty}$ can be  identified  with the set of all rigid  surjections
from $\om$ onto $\om$.
A surjection 
 $f:\om\ra\om$ 
 is {\em rigid} if  for each $i\in\om$, 
 $\min(f^{-1}\{i\})<\min(f^{-1}\{i+1\})$.
Given $E \in \mathcal{E}_{\infty}$, 
define $f_E(0)=0$. 
For $n \ge 1$,
 assuming that   for every $i \in n$ we have defined $f_{E}(i)$,
if  there is an $i \in n$ such that $n$ and $i$ belong to the same equivalence class of $E$, then  let $f_{E}(n)=f_{E}(i)$;
 otherwise, let  $f_{E}(n)=\max\{f_{E}(i): i \in n \}+1$. 
 Conversely, given  a rigid surjection 
 $h: \omega \rightarrow \omega$,
 the sets $h_{i}^{-1}(i)$ form a partition to $\omega$;
  we will denote this partition as $E_{h} \in \mathcal{E}_{\infty}$.
Given rigid surjections
 $g,h$, note that $E_{g} \leq E_{h}$ if and only if there exists a rigid surjection $f$ such that $f\circ h = g$ if and only if for each pair 
  $m<n < \omega$,
$h(n)=h(m)$ implies  $g(n)=g(m)$.

In \cite{matet1986}, Matet  studied the partial order $(\mathcal{E}_{\infty},\leq)$ as a lattice, and proved the following about the partial order $\le^*$, which was given in Definition
 \ref{defn.Mijares*}.

\begin{theorem}[Matet, \cite{matet1986}]
$(\mathcal{E}_{\infty},\leq^{\ast})$ is a $\sigma$-closed partial order.
\end{theorem}


\section{Ramsey degrees for ultrafilters associated to Ramsey spaces}\label{sec:degrees}

As seen in the previous section, many $\sigma$-closed forcings generating ultrafilters  of interest  have been shown to contain topological Ramsey spaces as dense subsets.
The initial purpose for constructing those  new Ramsey spaces  was to find the exact Rudin-Keisler and Tukey structures below those ultrafilters.
The Abstract Ellentuck Theorem  proved to be  vital to those investigations, which 
have  been the subject of work in \cite{R1},
\cite{Rn},
\cite{fraisseClasses},
\cite{highDimensional}, and
\cite{DobrinenJML16};
 the paper 
\cite{Dobrinen/Seals} provides an overview those results.

 In this section we develop a  
 general 
 method,
 utilizing  Theorem \ref{NashWilliams},
 to calculate Ramsey degrees for ultrafilters forced by  topological Ramsey  spaces with certain properties, which we call Independent Sequences of Structures,  discussed below.
 These properties are satisfied by 
the Ellentuck space, the spaces $\mathcal{R}_n$, $1\le n<\om$ (see Subsection 
\ref{subsec4.1}), and 
the  spaces generated by \Fraisse\ classes with the Ramsey property (see Subsection \ref{subsec.Fraisse}).
 Some of these ultrafilters are well-known and some are new, having  arisen  during  investigations  discussed in the previous paragraph.
This method also provides
simple, direct proofs for  some known Ramsey  degrees,  in particular,
 the ultrafilters  $\mathcal{U}_n$ 
 of Laflamme in \cite{Laflamme} mentioned in Subsection \ref{subsec4.1}.
 Without loss of generality (by restricting below some member of the space if necessary), we will assume all
 topological Ramsey spaces $\mathcal{R}$  contain a strongest member, denoted by $\mathbb{A}$.

 
 \subsection{A general method for Ramsey degrees for ultrafilters associated to Ramsey spaces composed of  independent sequences of structures}\label{subsec.5.1}

 The Ramsey spaces associated with the ultrafilters of Baumgartner-Taylor, Blass, and Laflamme mentioned in Section 4 all have the following property. 

\begin{definition}[Independent Sequences of Structures (ISS)]\label{defn.ISS}
We say that a   topological Ramsey space 
$(\mathcal{R}, \leq, r)$ has {\em Independent Sequences of Structures (ISS)}  if and only if the following hold:
There are relations $R_l$, $l<L$ for some fixed finite integer $L$,  where $R_0$ is a linear order,  and the domain of a structure $S$ with these relations is denoted $\dom(S)$.
The largest member $\mathbb{A}$ in $\mathcal{R}$ is a sequence $\lgl \mathbb{A}(i):i<\om\rgl$ such that each $\mathbb{A}(i)$ is a finite structure with relations $R_l$, $l<L$.
For $i<i'$, the domains of $\mathbb{A}(i)$  and $\mathbb{A}(i')$ are disjoint, and 
there are no relations between them, hence the {\em independence} of the sequence of structures.
Each  member $B \in \mathcal{R}$ can be identified with a sequence $\langle B(i): i \in \omega \rangle$ where each $B(i)$ is isomorphic to $\mathbb{A}(i)$, 
and moreover, there is a strictly increasing sequence $(k_i)_{i<\om}$ such that each $B(i)$ is an  induced 
substructure of $\mathbb{A}(k_i)$.
For $B,C\in\mathcal{R}$, $C\le B$ if and only if each $B(n)$ is a substructure of some $B(i_n)$  for some  strictly increasing sequence $(i_n)_{n<\omega}$.
The members of $\mathcal{AR}_m$ are simply initial sequences of  length $m$ of members of $\mathcal{R}$:
for $B\in\mathcal{R}$, 
$r_{m}(B):=\langle B(i): i < m\rangle$. 
We require that  $\dom(\mathbb{A}(0))$ is a singleton;
hence the members of $\mathcal{AR}_1$ are singletons.

For each  $m <  \omega$ and $a,b \in \mathcal{AR}_{m}$ there exists a unique (because of $R_0$) isomorphism $\varphi_{a,b}:a \rightarrow b$. 
We will write $\varphi$ to denote $\varphi_{a,b},$ when $a$ and $b$ are obvious. 
 If $X \in \mathcal{R}$ and $s \in [\mathcal{AR}_1\re X]^{n}$ for some $n \in \omega$, 
 we think of $s$ with the structure inherited by $X$.
Thus, if $k_0,\dots,k_m$ are those indices such that $x\cap X(k_i)\ne\emptyset$ for each $i\le m$,
then $s=\lgl s_i:i\le m\rgl$, where each $s_i$  is the structure on $\dom(s)\cap \dom(X(l_i))$ with the substructure inherited from $X(l_i)$.
Since each $X(l_i)$ is a substructure of some $\mathbb{A}(k_i)$,
each $s_i$ is also  the structure on $\dom(s)\cap \dom(\mathbb{A}(k_i))$ with the substructure inherited from $\mathbb{A}(k_i)$.
For  $s,t \in [\mathcal{AR}_{1}]^{n}$ we say that $s$ and $t$ are \emph{isomorphic}, and write $s\cong t$,  if  for some $m$, $s=\lgl s_i:i\le m\rgl$ and $t=\lgl t_i:i\le m\rgl$, and each $s_i$ is isomorphic to $t_i$. 
For $t \in [\mathcal{AR}_{1}]^{n},$ the \emph{isomorphism class} of $t$ is  the collection of substuctures $s \in [\mathcal{AR}_{1}]^{n},$ such that $s$ and $t$ are isomorphic.
\end{definition}

From now on, for $X\in\mathcal{R}$, we shall simply write $[X]^n$ instead of $[\mathcal{AR}_1\re X]^n$.

\begin{definition}[ISS$^+$]\label{defn.ISS+}
Let $\mathcal{R}$ be a space with the ISS.  We say that  $\mathcal{R}$  satisfies the ISS$^+$ if additionally, the following hold:
\begin{itemize}
\item[a)] 
There is some  $X \in \mathcal{R}$ such that for any two isomorphic members $u,v\in  [X]^{n}$, there exist $m \in \omega$,  $a,b \in \mathcal{AR}_{m}$, $s \in [a]^{n}$ isomorphic to $u$ and $t \in [b]^{n}$ isomorphic to $v$ such that  $\varphi(s)=t.$

\item[b)]
For every $n \geq 2,$ there exists an $m \in \omega$ such that for every $X \in \mathcal{R}$ and for every $s \in [\mathcal{AR}_{1}]^{n},$ there exists some $t \in [r_{m}(X)]^{n}$ such that $s$ and $t$ are isomorphic.
\end{itemize}
\end{definition}

For spaces with the ISS, the $\sigma$-closed partial order  $\le^*$ from  Definition \ref{defn.Mijares*}
 is simply $\sse^*$.
Given a generic filter $G$ forced by $(\mathcal{R},\sse^*)$, 
the
ultrafilter $\mathcal{U}_{\mathcal{R}}$ on  base set $\mathcal{AR}_1$  generated by $G$ was presented in 
Definition \ref{defn.U_R}.
If  $a \in \mathcal{AR}$ and $A \in \mathcal{R}$, we will write $[a]^{n}$ to denote $[\mathcal{AR}_{1}\upharpoonright a]^{n}$.

  \begin{definition}\label{defn.RamseydegreeG}
Given a  topological Ramsey space $(\mathcal{R},\le,r)$,
for $n\ge 1$,   define 
 $$
 t(\mathcal{R},n)
 $$ 
 to be the least  number $t$, if it exists,  such that for each $l\ge 2$ and each coloring $c:[\mathcal{AR}_1]^n\ra l$,
 there is a member $X\in \mathcal{R}$ such that 
 the restriction of $c$ to $[X]^n$ takes no more than $t$ colors.
  \end{definition}

  \begin{lemma}\label{fact.Important}
  Given a topological Ramsey space $(\mathcal{R},\le,r)$, 
  the ultrafilter $\mathcal{U}_{\mathcal{R}}$ generated by any
 generic filter $G$ forced by $(\mathcal{R},\le)$ satisfies 
 $$t(\mathcal{U}_{\mathcal{R}},n)=t(\mathcal{R},n)$$
  for each $n\ge 2$. 
 \end{lemma}
 
 This follows immediately by genericity of $G$. 
 Thus, finding the Ramsey degrees for topological Ramsey spaces is equivalent to finding the Ramsey degrees for their forced ultrafilters.

\begin{definition}\label{defn.k(R,n)}
If $(\mathcal{R}, \leq, r)$ is a topological Ramsey space satisfying the ISS$^+$, define $\bk(\mathcal{R},n)$ to be  the number of isomorphism classes for substructures $b \in [\mathcal{AR}_{1}]^{n}$ such that $b$ is a substructure of $\mathbb{A}(i)$ for some $i \in \omega$.
\end{definition}

Notice that b) of Definition \ref{defn.ISS+} guarantees that,  for each $n$,  $\bk(\mathcal{R},n)$  is finite.

\begin{lemma}\label{lem.tupperbd}
If $(\mathcal{R}, \leq, r)$ is a topological Ramsey space with  ISS$^+$, then  for each $n\ge 1$, 
\begin{equation}\label{eq.upperbd}
t(\mathcal{R},n) \leq \sum_{1 \leq q \leq n}   \ \sum_{j_{1}+...+j_{q}=n} 
\  \prod_{1 \leq i \leq q} \,  \bk(\mathcal{R},j_{i	}).\end{equation}
\end{lemma}

\begin{proof}
Fix   $n\ge 1$ and $p\ge 2$,   and let  $c:[\mathcal{AR}_{1}]^{n} \rightarrow p$ be a coloring. 
Let 
 $\tilde{k}$ denote the right hand side of the inequality in (\ref{eq.upperbd}), and 
 define 
$$
D=\{ X \in \mathcal{R}:|c''[X]^{n}| \leq \tilde{k}\}.
$$ 
We will prove that $D$ is dense in $\mathcal{R}.$ Let $\{\mathfrak{s}_{k}:k<\bk(\mathcal{R},n)\}$ be the collection of isomorphism classes for members of $[\mathcal{AR}_{1}]^{n}.$ 
Let $m$ be the least natural number such that for all $A \in \mathcal{R}$ and $k \in \tilde{k}$, there is some member of the isomorphism class $\mathfrak{s}_{k}$ contained in $r_m(A)$. 
Fix $a \in \mathcal{AR}_{m}$, and linearly order the members of  $[a]^{n}$ as $\{u_{l}:l < L\}$, where each $u_l$ is considered as a sequence of structures and 
where $L$ is the number of $n$-sized subsets of $a$. 
By the ISS$^+$, 
for every $b \in \mathcal{AR}_{m}$ there is an isomorphism $\varphi_{a,b} : a\rightarrow b$. 
Note that
 $\{\varphi_{a,b}(u_{l}): l < L\}$ is an enumeration for $[b]^{n}$ preserving the structure, 
 so for each $l<L$, $u_l$ is isomorphic to $\varphi_{a,b}(u_{l})$.

Let $\mathcal{I}=\space{}^{L}p.$
 For every $\iota \in \mathcal{I},$ define $$\mathcal{F}_{\iota}=\{b \in \mathcal{AR}_{m}:(\forall l \in L) \space c(\varphi_{b}( u_{l}))=\iota(l)\}.$$ 
Let  $A \in \mathcal{R}$ be given.
 Since $\mathcal{AR}_{m}$ is a Nash-Williams family and $\mathcal{AR}_{m}=\bigcup_{\iota\in \mathcal{I}} \mathcal{F}_{\iota},$ by Theorem \ref{NashWilliams} there are  $B \leq A$ and $\iota \in \mathcal{I}$ such that $\mathcal{AR}_{m}\upharpoonright B \subseteq \mathcal{F}_{\iota}.$ 
 Therefore, for every $b \in \mathcal{AR}_{m}\upharpoonright B$ and for every $l < L,$ $c( \varphi_{b}( u_{l}))=\iota(l).$  Hence $|c''[\mathcal{AR}_{m}\upharpoonright B]^{n}| \leq L$.
Now suppose that  $i< j < L$,   $u_{i}$  and $u_{j}$ are isomorphic. 
By a) of ISS$^+$,   there exist $b,d \in \mathcal{AR}_{m}\upharpoonright B,$ $s \in [b]^{n}$ isomorphic to $u_{i}$, and $t \in [d]^{n}$ isomorphic to $u_{j}$ such that  $\varphi_{b,d}(s)=t.$ 
Therefore, $c(u_{i})=c(u_{j}).$

Thus, it remains to  count the number of isomorphism classes in $[r_{m}(\mathbb{A})]^n$.
Let $t$ be a member of $[B]^{n}$ and note that for at least one  $i \in \omega,$ the substructure  obtained by intersecting $t$ with $ B(i)$ is not empty. 
Let $q$ be the cardinality of $\{ i \in \om:t\cap B(i)\ne\emptyset\}$.
 Note that if $q=1$, then $t$ belongs to one of $\bk(\mathcal{R},n)$ different isomorphism classes. 
If $q \geq 2$, let $\{l_{i}:i< q\}$ be an increasing enumeration of those $l\in \om$ such that $t\cap B(l)\ne\emptyset$,
 and let $t_i$ denote the substructure on $t\cap B(l_i)$ inherited from $B(l_i)$. 
For each  $ i \in [1,q]$, let $j_{i}$ denote the  cardinality  of  $\mathcal{AR}_1\re (t\cap B(l_{i})).$ 
Note that $n=\sum_{1 \leq i \leq q}j_{i}$ and every $j_{i}<n,$  
and  each  $t_{i}$ belongs to one of $\bk(\mathcal{R},j_{i})$ isomorphism classes.
Hence,   $t$ belongs to one of    $\bk(\mathcal{R},j_{1})\times...\times \bk(\mathcal{R},j_{q})$  many 
equivalence classes.
Letting $q$ range from $2$ to its maximum possibility of $n$, there are 
\begin{equation}
\sum_{1< q \leq n} \ \sum_{l_{1}+...+l_{q}=n} \ \prod_{i < q} \bk(\mathcal{R},j_{i	})
\end{equation}
 different isomorphism classes, for $n$ sized substructures of $B$ that contain substructures from more than one block. 
Thus,   $|c''[B]|^{n} \leq \tilde{k}$;  hence  $B \in D.$ 
Thus, $D$ is a dense subset of $\mathcal{R}$.
\end{proof}

The following Lemma tells us that the right hand side in  equation (\ref{eq.upperbd})  is not just an upper bound but it is the Ramsey degree, which will mean that it is enough to calculate the number of different ismorphism classes of j-sized substructures of blocks of $\mathbb{A}$ to know the exact Ramsey degree.

\begin{lemma}\label{lem.lowerbd}
If $(\mathcal{R}, \leq, r)$ is a topological Ramsey space with ISS$^+$,
then 
\begin{equation}\label{eq.upperbd}
t(\mathcal{R},n) \geq \sum_{1 \leq q \leq n} \ \sum_{j_{1}+...+j_{q}=n} \ \prod_{1 \leq i \leq q} \bk(\mathcal{R},j_{i	}).
\end{equation}
\end{lemma}

\begin{proof}
As in the previous lemma, let $\tilde{k}$ denote the right hand side of  equation (\ref{eq.upperbd}).
In the proof Lemma \ref{lem.tupperbd},
we showed that there are $\tilde{k}$ isomorphism classes for members of $[\mathcal{AR}_{1}]^{n}.$ 
Let $c:[\mathcal{AR}_{1}]^{n}\rightarrow \tilde{k}$ be a coloring such that for every $s,t \in [\mathcal{AR}_{1}]^{n},$ $c(s)=c(t)$ if and only if $s$ and $t$ belong to the same isomorphism class.
 By b) of ISS$^+$, there exists $m \in \omega$ such that for every $X \in \mathcal{R},$  and every $s \in [\mathcal{AR}_{1}]^{n}$ there is some member of $[r_{m}(X)]^{n}$ isomorphic to $s.$ Then for every $X \in \mathcal{R},$ the set $[X]^{n}$ contains members of every isomorphism class, and hence, $|c''[X]^{n}|=\tilde{k}.$
\end{proof}

\begin{theorem}\label{numbers}
Let $(\mathcal{R}, \leq, r )$ be  a topological Ramsey space with ISS$^+$.
Then 
\begin{equation}
t(\mathcal{R},n)= \sum_{1 \leq q \leq n} \, \sum_{j_{1}+...+j_{q}=n} \, \prod_{1 \leq i \leq q} \bk(\mathcal{R},j_{i	}).
\end{equation}
\end{theorem}

\begin{proof}
This follows from Lemmas \ref{lem.tupperbd} and \ref{lem.lowerbd}.
\end{proof}


\subsection{Calculations of Ramsey degrees of  ultrafilters from spaces with the ISS$^+$}\label{subsec.5.2}

In this subsection, we will calculate the exact Ramsey degrees for several classes of   ultrafilters  which are forced by spaces with the ISS$^+$.
First, we will use Theorem \ref{numbers} to provide a streamlined calculation of the  Ramsey degrees  for the weakly Ramsey ultrafilters forced by Laflamme's forcing $\mathbb{P}_1$.
Indeed, Laflamme calculated these in Theorem 1.10 of \cite{Laflamme}
via a  three-page proof which shows its three-way equivalence  with a combinatorial property that $\mathbb{P}_1$ is naturally seen to possess, and an interesting
Ramsey property for analytic subsets of the Baire space  in terms of the forcing $\mathbb{P}_1$ remeniscent of work of Mathias and Blass for Ramsey ultrafilters.
The proof we present here is direct and short.

For understanding the following proof, 
first notice that $\mathcal{AR}_1$ consists of all 
single maximal branches in the tree $\mathbb{T}_1$, 
that is, a set of the form $\{\lgl\rgl,\lgl i\rgl, \lgl i,j\rgl\}$, where $i\in \om$ and $j\le i$.
Note that   given  $n \in \omega$ fixed,
 every two members of $\mathcal{AR}_{n}$ are isomorphic as subtrees of   $\mathbb{T}_{1}$. 
This is because if $a,b \in \mathcal{AR}_{n},$ 
then there is an isomorphism $\varphi_{a,b}:a\rightarrow b$ which  sends each  node of the tree $a$ to the node in the same position of the tree $b$. 
The following two figures show  members, $a$ and $b$,  of $\mathcal{AR}_{5}$.
The isomorphism $\varphi_{a,b}$ for these finite trees sends $\langle0,0\rangle$ to $\langle 20,15\rangle,$ $\langle 1,0\rangle$ to $\langle 30,23\rangle,$ $\langle 1,1\rangle$ to $\langle 30,28\rangle,$ $\langle 2,0\rangle$ to $\langle 50,48\rangle,$ etc.

	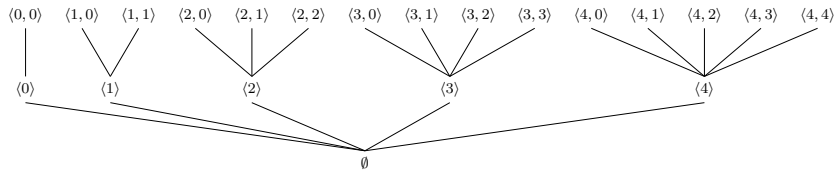
\begin{figure}[H]
		\centering
		{\footnotesize
			\begin{tikzpicture}[scale=.7,grow'=up, level distance=40pt,sibling distance=.2cm]
			\tikzset{grow'=up}
			\Tree [.$\emptyset$ [.$\langle0\rangle$ [.$\langle0,0\rangle$ ]] [.$\langle1\rangle$ [.$\langle1,0\rangle$ ][.$\langle1,1\rangle$ ] ] [.$\langle2\rangle$ [.$\langle2,0\rangle$ ] [.$\langle2,1\rangle$ ][.$\langle2,2\rangle$ ] ] [.$\langle3\rangle$ [.$\langle3,0\rangle$ ]  [.$\langle3,1\rangle$ ] [.$\langle3,2\rangle$ ] [.$\langle3,3\rangle$ ] ] [.$\langle4\rangle$ [.$\langle4,0\rangle$ ] [.$\langle4,1\rangle$ ] [.$\langle4,2\rangle$ ] [.$\langle4,3\rangle$ ][.$\langle4,4\rangle$ ] ]]
			\end{tikzpicture}
		}
		\caption{$a=r_{5}(\mathbb{T}_{1})$}
	\end{figure}

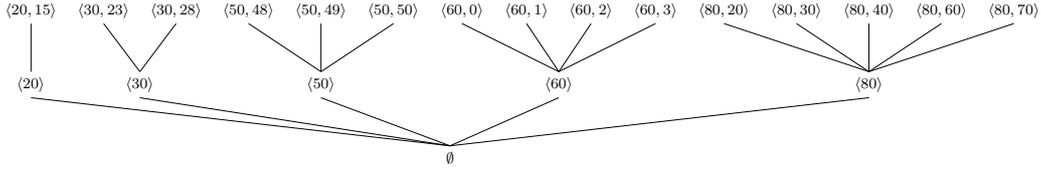
\begin{figure}[H]
	\centering
	{\footnotesize
		\begin{tikzpicture}[scale=.7,grow'=up, level distance=40pt,sibling distance=.2cm]
		\tikzset{grow'=up}
		\Tree [.$\emptyset$ [.$\langle20\rangle$ [.$\langle20,15\rangle$ ]] [.$\langle30\rangle$ [.$\langle30,23\rangle$ ][.$\langle30,28\rangle$ ] ] [.$\langle50\rangle$ [.$\langle50,48\rangle$ ] [.$\langle50,49\rangle$ ][.$\langle50,50\rangle$ ] ] [.$\langle60\rangle$ [.$\langle60,0\rangle$ ]  [.$\langle60,1\rangle$ ] [.$\langle60,2\rangle$ ] [.$\langle60,3\rangle$ ] ] [.$\langle80\rangle$ [.$\langle80,20\rangle$ ] [.$\langle80,30\rangle$ ][.$\langle80,40\rangle$ ] [.$\langle80,60\rangle$ ][.$\langle80,70\rangle$ ] ]]
		\end{tikzpicture}
	}
	\caption{Another member, $b$,  of $\mathcal{AR}_{5}$}
\end{figure}

Letting $b$ denote the member of $\mathcal{AR}_5$ in   Figure 5, note that $b(0)=\{\lgl\rgl, \lgl 20\rgl, \lgl 20,15\rgl\}$,
$b(1)=\{\lgl\rgl,\lgl 30\rgl,\lgl 30,23\rgl,\lgl30,28\rgl\}$, and so forth.

\begin{definition}
For every $n \geq 2,$ let $S_{n}=\{ x \in \space{}^{q}n:q \in [1,n],$ $\sum_{i<q} x(i)=n$ and $\forall i \in q(x(i) \neq 0) \}.$
\end{definition}

\begin{lemma}\label{Sn}
If $n \geq 2,$ then $|S_{n}|=\sum_{p < n}{n-1\choose p}=2^{n-1}.$
\end{lemma}

\begin{proof}
For $q \in [1,n]$ and $x \in \space{}^{q}n$ satisfying 
$\sum_{i<q} x=n$ and for every $i \in q,$ $x(i) \neq 0,$ 
let 
\begin{equation}
\psi(x)=\{ x(0)-1,x(0)+x(1)-1,...,\sum_{i<q-1}x(i)-1\}.
\end{equation}
Since every $x(i)\neq 0,$ $\psi(x)$ is a subset of $n-1$, so 
$\psi(x)$ is a member of $[n-1]^{q-1}.$ 
Notice that 
since $\sum_{i<q} x(i)=n$,
it follows that $x(q-1)=n-\sum_{i<q-1} x(i)$ is determined.

Note that the map $\psi:S_{n}\rightarrow \mathcal{P}(n-1)$ is one-to-one. 
Actually, $\psi$ is also an onto map. For every $p <n$ and $\{m_{0},...,m_{p-1}\} \in [n-1]^{p},$ a subset of $n-1$ with an increasing enumeration, $$\{m_{0},...,m_{p-1}\} = \psi (\langle m_{0}+1,m_{1}-m_{0},...,m_{p-1}-m_{p-2},n-1-m_{p-1}\rangle) $$ with $$\langle m_{0}+1,m_{1}-m_{0},...,m_{p-1}-m_{p-2},n-1-m_{p-1} \rangle \in (\space{}^{p+1}n) \cap S_{n}.$$ 
Then, $$|S_{n}|=\sum_{p< n}|[n-1]^{p}|=\sum_{p < n} \, {n-1 \choose p}  =2^{n-1}.$$
\end{proof}

\begin{corollary}\label{R1n}
	Let $\mathcal{U}_{1}$ be the weakly Ramsey  ultrafilter
forced by Laflamme's forcing $\mathbb{P}_1$,
equivalently, by  $(\mathcal{R}_{1}, \leq^{\ast})$.	 Then for each $n\ge 1$,   $t(\mathcal{U}_{1},n) = 2^{n-1} $.	
\end{corollary}

\begin{proof}
	Fix $n\ge 1$. 
	First, note that the space $\mathcal{R}_{1}$ satisfies ISS$^+$.
	 By Theorem \ref{numbers} we have $t(\mathcal{U}_{1},n)= \sum_{1 \leq q \leq n} \sum_{j_{1}+...+j_{q}=n} \prod_{1 \leq i \leq q} \bk(\mathcal{R}_{1},j_{i	}).$ 
	 Fix $1\le j\leq n$. 
	Given any  $j$-sized subsets   $s,t$   of $\mathcal{AR}_1$ 
	 such that $s \subset \mathbb{T}_1(i)$
	 and $t \subset \mathbb{T}_1(m)$
	  for some $i,m$, then $s$  and $t$ are  isomorphic.
	 Then for every $j \in [1,n],$ $\bk(\mathcal{R}_{1},j)=1.$ Therefore 
\begin{equation}
t(\mathcal{U}_{1},n)= \sum_{1 \leq q \leq n}\  \sum_{j_{1}+...+j_{q}=n}\  \prod_{1 \leq i \leq q} 1 =|S_{n}|=2^{n-1}.
\end{equation}
  Therefore $t(\mathcal{U}_{1},n) = 2^{n-1} .$
\end{proof}

Since $t(\mathcal{U}_{1},2) = 2, $ for every $c:[\mathcal{AR}_{1}]^{2} \rightarrow 3$ coloring there is an $X \in \mathcal{U}_{1}$ such that $|c\upharpoonright [X]^{2}| \leq 2$. If we indentify $\mathcal{AR}_{1}$ with $\omega$ then $\mathcal{U}_{1}$ is a weakly Ramsey ultrafilter in the sense of $\omega$. 
By Lemma 3, there is a coloring $c:[\mathcal{AR}_{1}]^{2} \rightarrow 2$ such that for every $X \in \mathcal{R}_{1}$, $|c\upharpoonright [X]^{2}| = 2$. 
Thus, we can see in a simple way that $\mathcal{U}_{1}$ is not a Ramsey ultrafilter.

Next, we will calculate Ramsey degrees for 
ultrafilters  $\mathcal{U}_k$ forced by Laflamme's forcings $\mathbb{P}_k$ from \cite{Laflamme}, $k\ge 2$.
As noted in the previous secion, the topological Ramsey 
space $\mathcal{R}_{k}$  forces the same ultrafilter as $\mathbb{P}_k$.  
The Ramsey degrees for $\mathcal{U}_k$  
are stated  in Theorem 2.2 of 
\cite{Laflamme}, but 
a concrete proof does not appear in that paper.
 Rather, Laflamme points out 
 that the proof entirely similar to, but combinatorially more complicated than  that of Theorem 1.10  in \cite{Laflamme}.
Here,  we present a straightforward 
proof based on the  Ramsey structure of $\mathcal{R}_{k}$.  
For the following proof, the set of first approximations $\{r_1(A):A\in\mathcal{R}_k\}$ are identified with the maximal  nodes of $\mathbb{T}_{k}$;  this also applies for every member of $\mathcal{R}_{k}$.
The downward closure of any maximal node in $\mathbb{T}_k$ recovers the tree structure below that node, so it suffices to work with the maximal  nodes in $\mathbb{T}_k$.

Recall Definition \ref{defn.k(R,n)} of $\bk(\mathcal{R},n)$ for a topological Ramsey space $\mathcal{R}$.
The next lemma uses the inductive nature of the construction of $\mathcal{R}_{k+1}$ from $\mathcal{R}_k$ to   show that each $\bk(\mathcal{R}_{k+1},n)$ can be deduced from the Ramsey degrees of $\mathcal{U}_k$.

\begin{lemma}\label{lem.inductionR_k}
For any  $k,n \ge 1$,  we have that
$\bk(\mathcal{R}_{k+1},n)=t(\mathcal{U}_{k},n).$
\end{lemma}

\begin{proof}
Let $s \in [\mathbb{T}_{k+1}]^{n}$ be such that $s \subset \mathbb{T}_{k+1}(l)$ for some $l \in \omega$.
Recall our convention that $s$ is a collection of maximal nodes in $\mathbb{T}_{k+1}$, so each node in $s$ is a sequence of length $k+2$.
 Note that since $s$ is contained in $\mathbb{T}_{k+1}(l)$, 
each  node in $s$ end-extends the sequence $\lgl l\rgl$. 
Let 
$t$ be  the set of sequences resulting by taking out the first member of every sequence in $s$; thus, letting ${}^{k+1}\om$ denote the set of sequences of natural numbers of  length $k+1$,  
\begin{equation}
t=\{ x\in {}^{k+1}\om: \lgl l\rgl^{\frown} x\in s \}.
\end{equation}
Then $t$ is an $n$-sized subset of $\mathbb{T}_{k}$.
  Since there are $t(\mathcal{U}_{k},n)$ isomorphism classes for $[\mathbb{T}_{k}]^{n},$ $\bk(\mathcal{R}_{k+1},n)=t(\mathcal{U}_{k},n).$
\end{proof}

\begin{lemma}\label{lem.RdegcalR_k}
Let $k\ge 1$ be given, and suppose that  
$\mathcal{U}_{k+1}$ is an $(\mathcal{R}_{k+1},\leq^{\ast})$-generic filter.
Then for each $n\ge 1$,
\begin{equation}
t(\mathcal{U}_{k+1},n)= 
\sum_{1 \leq q \leq n}\  \sum_{j_{1}+...+j_{q}=n}\  \prod_{1 \leq i \leq q} t(\mathcal{U}_{k},j_{i	}).
\end{equation}
\end{lemma}

\begin{proof}
This follows from 
Theorem \ref{numbers} and 
Lemma \ref{lem.inductionR_k}.
\end{proof}

\begin{theorem}\label{thm.RamseydegLaflamme}
Given $k,n\ge 1$,
   if $\mathcal{U}_{k}$ an ultrafilter forced by Laflamme's $\mathbb{P}_k$, or equivalently by 
    $(\mathcal{R}_{k},\leq^{\ast})$, then
 $t(\mathcal{U}_{k},n)= (k+1)^{n-1}$.
\end{theorem}

\begin{proof}
The proof is by induction on $k$ over all $n\ge 1$.
The case when $k=1$ is done by Corollary \ref{R1n}. Now we assume the conclusion for  a fixed $k\ge 1$ and prove it for $k+1$. 
By the Lemma \ref{lem.RdegcalR_k}, 
\begin{equation}
t(\mathcal{U}_{k+1},n)= \sum_{1 \leq q \leq n} \, \sum_{j_{1}+...+j_{q}=n} \, \prod_{1 \leq i \leq q} t(\mathcal{U}_{k},j_{i	}).
\end{equation}
 By inductive hypothesis $t(\mathcal{U}_{k},j_{i	})=(k+1)^{j_{i}-1}$.
  Then
\begin{equation}
t(\mathcal{U}_{k+1},n)= \sum_{1 \leq q \leq n} \ \sum_{j_{1}+...+j_{q}=n} \ \prod_{1 \leq i \leq q}(k+1)^{j_{i}-1}=\sum_{1 \leq q \leq n} \ \sum_{j_{1}+...+j_{q}=n}(k+1)^{n-q} .
\end{equation} 
By the proof of Lemma \ref{Sn},  
\begin{equation}\label{eq.13}
\sum_{1 \leq q \leq n} \
 \sum_{j_{1}+...+j_{q}=n} (k+1)^{n-q} =\sum_{1 \leq q \leq n}  {n-1 \choose q-1} (k+1)^{n-q}=\sum_{0 \leq p \leq n-1}  {n-1 \choose p}  (k+1)^{n-1-p}. 
\end{equation}
 By Newton's Theorem,
 the right hand side of equation (\ref{eq.13}) equals 
  $((k+1)+1)^{n-1}$.
  Therefore $t( \mathcal{U}_{k+1},n)= (k+2)^{n-1}.$
\end{proof}

Next, we calculate the Ramsey degree for pairs for the ultrafilters forced by Blass' $n$-square forcing and more generally, the hypercube Ramsey spaces $\mathcal{H}^n$.

\begin{corollary}
Let $\mathcal{V}_n$ be an $(\mathcal{H}^n,\leq^{\ast})$-generic filter. Then 
$$
t(\mathcal{V}_n,2) = 1+ \sum_{i=0}^{n-1} 3^i.
$$
In particular, $t(\mathcal{V}_2,2) = 5$, where $\mathcal{V}_2$ is the ultrafilter generated by Blass' $n$-square forcing.
\end{corollary}

\begin{proof}
By Theorem \ref{numbers},  we know that  for each $n\ge 2$, 
$$
t(\mathcal{V}_{n},2)= \sum_{1 \leq q \leq 2} \, \sum_{j_{1}+...+j_{q}=2} \, \prod_{1 \leq i \leq q} \bk(\mathcal{H}^{n},j_{i	})=\bk(\mathcal{H}^{n},2)+\bk(\mathcal{H}^{n},1).$$
Note that $\bk(\mathcal{H}^{n},1)=1$ because all the singletons are isomorphic.
Given $n\ge 2$, let  $\mathbf{A}_k$ be an $n$-hypercube with side length $k$. 
We will show that 
$\bk(\mathcal{H}^{n},2)= \sum_{i=0}^{n-1} 3^i$.

Fix $n=2$, and  take $a=(a_{0},a_{1}),\ b=(b_{0},b_{1}) \in \mathbf{A}_{k}$ for  any large enough $k \in \omega.$ 
Assume that $a$ lexicographically below  $b$,  according to the lexicographical order on $\omega \times \omega.$ 
 There are $4$ non-isomorphic options:
Either $a_0=b_0$ and $a_1<b_1$, or else 
 $a_{0}<b_{0}$ 
and any of the three relations 
$a_{1}<a_{1}$,
 $a_{1}=b_{1},$
or  $a_{1}>b_{1}$  holds.
Therefore $\bk(\mathcal{H}^{2},2)=4 = 3^0+3^1,$ so $t(\mathcal{U}_{2},2)=5.$

Now suppose that
$\bk(\mathcal{H}^{n},2)= \sum_{i=0}^{n-1} 3^i$.
Given $a=(a_0,\dots,a_n)$ and  $b=(b_0,\dots,b_n)$ 
in $\mathbf{A}_k$ for any large enough $k\in \om$, there are the following possibilities.
If $a_0=b_0$, then there are $\bk(\mathcal{H}^n,2)=\sum_{i=0}^{n-1} 3^i$ many possible relations between $(a_1,\dots,a_n)$ and $(b_1,\dots,b_n)$.
If $a_0<b_0$,  then for each $1\le i\le n$, there are three possible configurations for $a_i$ and $b_i$, namely $a_i<b_i$, $a_i=b_i$, or $a_i<b_i$.
Thus, there are $3^n$  many possible configurations for $(a_1,\dots, a_n)$ and $(b_1,\dots,b_n)$.
Therefore, 
$\bk(\mathcal{H}^{n+1},2)= \sum_{i=0}^{n} 3^i$.
\end{proof}

Finally, we calculate the Ramsey degrees for the
$k$-arrow, not $(k+1)$-arrow ultrafilters of Baumgartner and Taylor. 

\begin{corollary}
For $k\ge 2$, let $\mathcal{W}_k$ be 
the $k$-arrow, not $(k+1)$-arrow  ultrafilter of Baumgartner and Taylor. 
Then 
$$
t(\mathcal{W}_k,2) = 3.
$$
\end{corollary}

\begin{proof}
Let $\mathcal{A}_k$ denote the topological Ramsey space constructed from a generating sequence of finite graphs which have no $k$-cliques. 
As this space is dense in the Baumgartner-Taylor forcing, the two partial orders generate the same ultrafilters. 
By Theorem \ref{numbers},  we know that  for each $n\ge 2$, 
$$
t(\mathcal{A}_k,2)= \bk(\mathcal{H}^{n},2)+\bk(\mathcal{H}^{n},1).$$
As in the previous corollary, 
 $\bk(\mathcal{A}_k,1)=1$ because all the singletons are isomorphic.
 Note that 
 $\bk(\mathcal{A}_k,2)=2$, since pairs  of singletons in $\mathcal{AR}_1$ in this topological Ramsey space have two isomorphism classes: either a pair has an edge or else it has no edge between them. 
\end{proof}


\section{Ramsey degrees for ultrafilters forced by $
\mathcal{P}(\om^k)/\Fin^{\otimes k}$}\label{sec:hdegrees}

In this Section we will calculate Ramsey degrees of pairs for 
ultrafilters forced by 
$\mathcal{P}(\om^k)/\Fin^{\otimes k}$, for all $k\ge 2$.
In this section, let 
    $\mathcal{G}_1$  denote the ultrafilter  forced by $\mathcal{P}(\om)/\Fin$ and note that 
  $\mathcal{G}_1$ is a  Ramsey ultrafilter.
Recall from Subsection \ref{subsec.hdE}
that 
 $\Fin^{\otimes k}$  is a $\sigma$-closed ideal
 on $\om^k$, and that the Boolean algebras
  $\mathcal{P}(\om^k)/\Fin^{\otimes k}$ force ultrafilters $\mathcal{G}_k$ such that whenever $1\le j<k$, 
  the projection of $\mathcal{G}_k$ to the first $j$ coordinates of $\om^k$ forms an ultrafilter on $\om^j$ which is generic for 
$\mathcal{P}(\om^j)/\Fin^{\otimes j}$.
$\mathcal{G}_2$ garners much attention as it is a weak p-point which is not a p-point (see \cite{Blass/Dobrinen}).
To ease notation,  let 
 $\sse^{*k}$ denote  $\sse^{\Fin^{\otimes k}}$, and note that   $(\mathcal{E}_k,\sse^{*k})$ is forcing equivalent to $\mathcal{P}(\om^k)/\Fin^{\otimes k}$
 (see \cite{highDimensional}).
 In this section, 
 we use the  high dimensional Ellentuck spaces, $\mathcal{E}_k$, 
 to provide concise proofs of the Ramsey degrees $t(\mathcal{G}_k,2)$ for all $k\ge 2$.

\begin{definition}
Let $k \geq 2$ and $s,t,u,v \in \ome{k}$. 
Define the relation $\sim_{k}$ on pairs of $\ome{k}$ by $\langle s,t\rangle \sim_{k}\langle u,v\rangle  $ if and only if $s \prec t,$ $u \prec v$ and for every $ i,j \in k$ and $\rho \in \{=,<\},$ $s_{i} \, \rho \, t_{j} \longleftrightarrow u_{i} \, \rho \, v_{j}.$
\end{definition}

Observe that 
the set of $X\in \mathcal{E}_{k}$
satisfying 
\begin{equation}\label{eq.property}
\forall s,t \in X,\ 
\forall 1\le i<k\ (s_i=t_i\longrightarrow s_{i-1}=t_{i-1})
\end{equation}
is open dense in $\mathcal{E}_{k}$.
Let $\mathcal{D}_k$ denote the set of $X\in\mathcal{E}_k$ satisfying (\ref{eq.property}).
From now on we will work only with members of 
$\mathcal{D}_{k}$.
Let ${\mathbf k}(\mathcal{E}_{k},2)$ be the number of equivalence classes from $\sim_{k}$ on pairs of $\mathcal{D}_{k}$.  
We will first calculate this number, and then  show in Theorem \ref{thm.RdegGk}
that it actually is the Ramsey degree of $\mathcal{G}_k$ for pairs.

\begin{lemma}\label{lem.bfkEk}
For every $k \geq 2$, ${\mathbf k}(\mathcal{E}_{k},2)=\sum_{i< k}3^{i}.$
\end{lemma}

\begin{proof}
The proof will be by induction on $k\ge 2$.
Fix $k=2$, and fix some $X\in \mathcal{D}_2$.
Fix   $s,t \in  X$ such that $t \prec s$.
Then  $s_{0}< s_{1},$ $t_{0}< t_{1}$ and $t_1<s_1$, where $s=(s_{0},s_{1}),$ $t=(t_{0},t_{1})$.
There are four possibilities for ordering  the entries of $s$ and $t$:
\begin{itemize}
\item[(i)]   $t_{0}=s_{0}$ and $t_1<s_1.$
\item[(ii)] $t_0 < t_{1} < s_{0} < s_1.$
\item[(iii)]  $t_{0}< s_{0}<t_1<s_1.$
\item[(iv)]  $s_{0}< t_{0}<s_1$ and $t_1<s_1.$
\end{itemize}
Note that these four options are not isomorphic and for every $X \in \mathcal{D}_{2},$ $[X]^{2}$  contains all pairs in all four types. 
 Therefore ${\bf k}(\mathcal{E}_{2},2)=4=1+3.$

Now assume the conclusion holds  for some fixed $k \geq 2$; we will prove that conclusion holds for $k+1.$ 
Fix some $X\in\mathcal{D}_{k+1}$ and 
take $s,t \in X $ such that $t \prec s$.
Then $s=(s_{0},...,s_{k-1},s_{k}),$ $t=(t_{0},...,t_{k-1},t_{k})$ with $s_{0}<...< s_{k},$ $t_{0}<...< t_{k}$ and $t_{k}<s_{k}$.
 There are four options for ordering last two members of sequences $s$ and $t$:
\begin{itemize}
\item[(i)] 
$t_{k}< s_{k-1}.$ In this case $t_{k-1}<s_{k-1}.$ Then $t\upharpoonright k \prec s\upharpoonright k$ and the pair $\langle s\upharpoonright k,t\upharpoonright k\rangle $ lies in one of ${\bf k}(\mathcal{E}_{k},2)$ possible equivalence classes. Therefore, there are ${\bf k}(\mathcal{E}_{k},2)$ options for ordering members of sequences $s$ and
 $t$.
\item[(ii)] 
$t_{k}> s_{k-1}$ and $t_{k-1}< s_{k-1}.$ In this case $t\upharpoonright k \prec s\upharpoonright k$ and the pair $\langle s\upharpoonright k,t\upharpoonright k\rangle$ lies in one of ${\bf k}(\mathcal{E}_{k},2)$ possible equivalence classes. Therefore, there are ${\bf k}(\mathcal{E}_{k},2)$ options for ordering members of sequences $s$ and $t.$
\item[(iii)] 
 $t_{k}> s_{k-1}$ and $t_{k-1}> s_{k}.$ In this case $s\upharpoonright k \prec t\upharpoonright k$ and the pair $\langle t\upharpoonright k,s\upharpoonright k \rangle$ lies in one of ${\bf k}(\mathcal{E}_{k},2)$ possible equivalence classes. Therefore, there are ${\bf k}(\mathcal{E}_{k},2)$ options for ordering members of sequences $s$ and $t.$
\item[(iv)] 
 $t_{k-1}=s_{k-1}.$ In this case,  $t_i=s_i$ for all $i\le k-1$.
\end{itemize}
Then there are $3{\bf k}(\mathcal{E}_{k},2)+1$ equivalence classes for the relation $\sim_{k+1}.$ 
Since ${\bf k}(\mathcal{E}_{k},2)=\sum_{i< k}3^{i},$ then 
\begin{equation}
{\bf k}(\mathcal{E}_{k+1},2)=3(\sum_{i< k}3^{i})+1=\sum_{i< k+1}3^{i}.
\end{equation}
\end{proof}

%


\begin{theorem}\label{thm.RdegGk}
	For every $k \in \omega$ such that $k \geq 2$, $t(\mathcal{G}_{k},2) =\sum_{i< k}3^{i}.$
\end{theorem}

\begin{proof}
Let $c:[\ome{k}]^{2}  \rightarrow t(\mathcal{G}_{k},2) $ be such that $c(p)=c(q)$ if and only if $p \sim_{k} q.$ For each  $X \in \mathcal{E}_{k}$,
 $[X]^{2}$ contains members of every equivalence class.
Thus,  by Lemma \ref{lem.bfkEk},
$|c\upharpoonright [X]^{2}| = t(\mathcal{G}_{k},2)\geq \sum_{i< k}3^{i}.$

Now we want to prove that $t(\mathcal{G}_{k},2)\leq \sum_{i< k}3^{i}.$	
Let $r\ge 1$ be a natural number and $c:[\ome{k}]^{2}  \rightarrow r $ be a coloring.
 Let 
 \begin{equation}
 \mathcal{D}=\{ Y \in \mathcal{E}_{k}: |c\upharpoonright [Y]^{2}| \leq {\bf k}(\mathcal{E}_{k},2) \}.
 \end{equation}
  We will prove that $\mathcal{D}$ is a dense subset of $\mathcal{G}_{k}$. 
  Let $X \in \mathcal{G}_{k}\cap\mathcal{D}_k$, and 
  let $m$ be a natural number such that $r_{m}(X)$ contains pairs of every equivalence relation of $\sim_{k}$.

Fix $a \in \mathcal{AR}_{m}$ and fix  an order for $[a]^{2}=\{ p^{a}_{l}:l < L \}$ with the induced substructure, where $L$ is the number of pairs of sequences that belong to $a$. 
For each $b \in \mathcal{AR}_{m}$ enumerate pairs of $b$ as $[b]^{2}=\{ p^{b}_{l}:l < L \}$ such that for every $l< L,$ $ p^{a}_{l} \sim_{k} p^{b}_{l}$.
Let $\mathcal{I}={}^{L}r$.
 For every $\iota \in \mathcal{I}$, define 
 \begin{equation}
 \mathcal{F}_{\iota}=\{b \in \mathcal{AR}_{m}:(\forall l <L)\, c(p^{b}_{l})=\iota(l)\}.
 \end{equation}
 Since $\mathcal{AR}_{m}$ is a Nash-Williams family and $\mathcal{AR}_{m}=\bigcup_{\iota\in \mathcal{I}} \mathcal{F}_{\iota},$ by Theorem \ref{NashWilliams} there exist $Y \leq X$ and $\iota \in \mathcal{I}$ such that $\mathcal{AR}_{m} \upharpoonright Y \subseteq \mathcal{F}_{\iota}.$ 
 Therefore, for every $b \in \mathcal{AR}_{m} \upharpoonright Y$ and for every $l < L,$ $c(p^{b}_{l})=\iota(l).$  
 Hence $|c''[\mathcal{AR}_{m}\upharpoonright Y]^{2}| \leq L.$

Note that if $i,j$ are such that $p^{a}_{i} \sim_{k} p^{a}_{j}$, then there exist $A,B \in \mathcal{AR}_{m}\upharpoonright Y$ and $l < L$ such that $p^{a}_{i} \sim_{k}p^{A}_{l},$ $p^{a}_{j} \sim_{k}p^{B}_{l}$.
This  implies that if $p^{a}_{i} \sim_{k}p^{a}_{j}$ with $i< j \leq L$, then $c(p^{a}_{i})=(p^{a}_{j}).$ Since there are ${\bf k}(\mathcal{E}_{k},2)$ different equivalence classes, and $[Y]^{2}$ contains pairs of every equivalence class, we obtain that $|c''[\mathcal{AR}_{1}\upharpoonright Y]^{2}| \leq {\bf k}(\mathcal{E}_{k},2).$ 
Therefore $Y \in D$; hence,  $D$ is a dense subset of $\mathcal{E}_{k}$.
It follows from  
Lemma \ref{lem.bfkEk} $t(\mathcal{G}_k,2)\le \sum_{i<k} 3^i$.
\end{proof}


\section{Pseudointersection and tower numbers for topological Ramsey spaces}\label{sec:Cardinals}

In this Section we investigate the pseudointersection and tower numbers for   several classes of topological Ramsey spaces  and their relationships to 
$\mathfrak{p}$ (recall Definition \ref{defn.SFIP}).
The notion of strong finite intersection property, and hence, also the pseudointersection and tower numbers  from Section \ref{sec:Background} can be extended to any $\sigma$-closed partial order, in particular, to topological Ramsey spaces.

\begin{definition}\label{defn.sfipR}
We say that $(\mathcal{R},\le,\le^*,r)$ is a {\em $\sigma$-closed  topological Ramsey space} if 
 $(\mathcal{R},\le, r)$ is  a topological Ramsey space
 and 
$\le^*$ is  a $\sigma$-closed order on $\mathcal{R}$ coarsening $\le$ such that $(\mathcal{R},\le)$ and $(\mathcal{R},\le^*)$ have isomorphic separative quotients.
Let  $\mathcal{F}$ be a subset of $\mathcal{R}$.
\begin{enumerate}
\item 
We say that $\mathcal{F}$ has the \emph{strong finite intersection property} (SFIP) if for every finite subfamily $ \mathcal{G} \subseteq \mathcal{F}$, there exists $Y \in \mathcal{R}$ such that for each $X \in \mathcal{F}$, $Y \le^* X$.
\item 
$Y \in \mathcal{R}$ is  called a \emph{pseudointersection} of the family $\mathcal{F}$ if for every $X \in \mathcal{F}$, $ Y \le^* X .$
\end{enumerate}
\end{definition}

\begin{definition}\label{defn.p_Rt_R}
Let $(\mathcal{R},\le,\le^*, r)$ be a $\sigma$-closed topological Ramsey space.
\begin{enumerate}

\item The \emph{pseudointersection number} $\mathfrak{p}_{\mathcal{R}}$    is the smallest cardinality of a family $\mathcal{F} \subseteq \mathcal{R}$ which has the SFIP but does not have a pseudointersection.

\item 
We say that $\mathcal{F}$  is a \emph{tower} if it is linearly ordered by $\ge^*$
 and has no pseudointersection.
The \emph{tower number} $\mathfrak{t}_{\mathcal{R}}$ is the smallest cardinality of a tower of $(\mathcal{R},\le^*)$.
\end{enumerate}
\end{definition}

Note that for every topological Ramsey space $\mathcal{R},$ $\mathfrak{p}_{\mathcal{R}} \leq \mathfrak{t}_{\mathcal{R}}$.
A recent groundbreaking result of Malliaris and Shelah shows that $\mathfrak{p}=\mathfrak{t}$ (see \cite{pt} and \cite{Malliaris/ShelahJAMS16}).
It  is not clear at present whether their work  implies that  
$\mathfrak{p}_{\mathcal{R}}=\mathfrak{t}_{\mathcal{R}}$ for all   Ramsey spaces  with some 
$\sigma$-closed partial order.
For all the spaces considered in this section, we will show that they are indeed equal.

As in previous sections,  we will assume that  
$\mathcal{AR}$ is  countable. 
If this is not the case,   we tacitly work on $\mathcal{AR}\upharpoonright X$ for some $X \in \mathcal{R},$ which is countable by axiom \bf A.2\rm.

\subsection{Pseudointersection and tower numbers for several classes of  topological Ramsey spaces}\label{subsec.IEP}

The following property is satisfied by many topological Ramsey spaces, including several discussed in Section \ref{sec:tRs&uf}.

\begin{definition}[IEP]\label{defn.IEP}
We say that a topological Ramsey space $(\mathcal{R},\le,r)$ has the 
 {\em Independent Extension Property (IEP)} if the following hold:
Each $X \in \mathcal{R}$ is a sequence of the form 
$ \langle X(n):n \in \omega \rangle$ such that
 for every $n \in \omega$, $r_{n}(X)= \langle X(i):i < n \rangle$, and each $X(i)$ is a finite set, possibly, but not necessarily, with some relational structures on it. 
 Furthermore, 
 for every $X \in \mathcal{R}$, $k \in \omega,$ and $s \in \mathcal{AR}_{k},$ there exist $m \in \omega$ and $s(k) \subseteq X(m)$ such that $s ^{\frown} s(k) \in \mathcal{AR}_{k+1}$ and $s\sqsubseteq s ^\frown s(k)$.
\end{definition}

\begin{theorem}\label{pleqpR}
	Let $(\mathcal{R}, \leq,\le^*,r)$ be a $\sigma$-closed topological Ramsey space 
	 with the IEP, and suppose that  $\mathcal{R}$  closed in $\mathcal{AR}^{\omega}$.
Then $\mathfrak{p}\leq  \mathfrak{p}_{\mathcal{R}}$.
\end{theorem}

\begin{proof}
Recall that by Bell's result, $\mathfrak{p}=\mathfrak{m}(\sigma$-centered).
	Let $\kappa < \mathfrak{m}(\sigma$-centered) be given, and let $\mathcal{F}= \{X_{\alpha} : \alpha < \kappa \} \subseteq \mathcal{R}$ be a family with the SFIP. 
	Define $\mathbb{P}$ to be the collection of all ordered pairs $ \langle s, E \rangle $ such that $s \in \mathcal{AR}$ and $E \in [\kappa]^{< \omega}$. 
	Since $\mathcal{R}$ satisfies the IEP, define some shorthand notation as follows:
For $s,t\in\mathcal{AR}$ with $s\sqsubseteq t$, 
note that $s=\lgl s(i):i<m\rgl$ and $t=\lgl t(i):i<n\rgl$ for some $m\le n$.
Then let 
\begin{equation}
t / s=\lgl t(i): m\le i<n\rgl.
\end{equation}
 Define a partial order $\le$ on $\mathbb{P}$ as follows:
	Given $\langle s, E \rangle, \langle t,F \rangle \in \mathbb{P}$, let  $ \langle t, F \rangle \leq \langle s, E \rangle $ iff $s \sqsubseteq t$, $E \subseteq F$, and there exists $X \in \mathcal{R}$ such that for every $\alpha \in E$, $X \leq X_{\alpha}$ and $t / s   \subseteq X$, which means that for every $i \in [|s|,|t|)$ there exists $l$ such that $t(i) \subseteq X(l)$.

For every $s \in \mathcal{AR}$, define $\mathbb{P}_{s}=\{ \langle s,E \rangle : E \in [\kappa]^{< \omega} \}$. Note that
$\mathbb{P}_{s}$ is centered. Since $\mathcal{AR}$ is countable, $\mathbb{P}= \bigcup_{s \in \mathcal{AR}} \mathbb{P}_{s}$ is a $\sigma$-centered partial order. Given $\alpha < \kappa$ and $m \in \omega$ let $$D_{\alpha , m} = \{ \langle s,E \rangle \in \mathbb{P} : \alpha \in E \, \mathrm{and} \, | s | > m \}.$$

\begin{claim}$1$.
$D_{\alpha , m}$ is dense.
\end{claim}

\begin{proof}
Fix $\langle t,F \rangle \in \mathbb{P}$. 
If $|t|>m,$ then the pair $\langle t,F \cup \{\alpha\} \rangle \leq \langle t,F \rangle$ is in $D_{\alpha,m}$. 
If $|t| \leq m$, since $\mathcal{F}$ has the SFIP, there exists $X \in \mathcal{R}$ such that for every $\beta \in F$, $X \leq X_{\beta} $. Let $i = |t|$. 
By the IEP, there is some $u \in \mathcal{AR}_{m+1}$ such that $u$ extends $t$ into $X.$ 
That is, there is a strictly increasing sequence $l_{i}<\dots <l_{m}$ and substructures $u(j) \in \mathcal{R}(j)\upharpoonright X(l_{j}),$ for each $j \in [i,m]$, such that $u= t ^{\frown} \langle u(j):i \leq j \leq m \rangle$
is a member of 
$\mathcal{AR}_{m+1}$. 
Let $E = F \cup \{ \alpha \}$. By the choice of $u,$ it follows that $\langle u,E \rangle \leq \langle t,F \rangle$ and $\langle u,E \rangle \in D_{\alpha,m}.$
\end{proof}

Let $\mathcal{D} = \{D_{\alpha,m} : \alpha \in \kappa$ and $m \in \omega \}$ and let $G$ be a filter on $\mathbb{P}$ meeting each dense set in $\mathcal{D}$. Let $X_{G} =  \bigcup\{s : \exists E \in [\kappa]^{< \omega} ( \langle s,E \rangle \in G ) \}$. Since $\mathcal{R}$ is closed under $\mathcal{AR}^{\omega},$ $X_{G} \in \mathcal{R}$.

\begin{claim}$2.$
$\forall \alpha \in \kappa$, $X_{G} \leq^{\ast} X_{\alpha}$.
\end{claim}

\begin{proof}
	 Fix $\alpha < \kappa$ and some $\langle s,E \rangle \in G \cap D_{\alpha,0}$. We will prove that for every $m \geq |s|,$ $r_{m}(X_{G})  \   s \subseteq X_{\alpha}$. Let $m > |s|$ and $ \langle t,F \rangle \in G \cap D_{\alpha,m}$ be given. 
	 Since $D_{\alpha,m}$ is an open dense subset of $\mathbb{P}$ and $G$ is a filter, there is a condition $\langle u,H \rangle$ below both $\langle s,E \rangle$ and $ \langle t,F \rangle$ 
	 such that  $\langle u,H \rangle\in G \cap D_{\alpha,m}.$ Note that $u=r_{|u|}(X_{G}).$  Since $\langle u,H\rangle \leq \langle s,E\rangle$ and $\alpha \in E$, there exists $X \in \mathcal{R}$ such that $X \leq X_{\alpha}$ and $u / s \subseteq X \leq X_{\alpha}$. It follows from $|u|>m$ that $r_{m}(X_{G}) / s \subseteq X_{\alpha}.$
	\end{proof}
	
	Therefore, 
	$\kappa < \mathfrak{p}_{\mathcal{R}}$; and hence,
	$\mathfrak{m}(\sigma$-centered$)\le \mathfrak{p}_{\mathcal{R}}$. \end{proof}

This immediately leads to the following corollary for all the topological Ramsey spaces from Subsections \ref{subsec4.1},
\ref{subsec.Fraisse}, and \ref{subsec.FINk}.
For each of these spaces, the relevant $\sigma$-closed partial order $\le^*$ from Definition \ref{defn.Mijares*} is exactly the mod finite partial order, $\sse^*$.

\begin{corollary}\label{cor.m<pR} 
Let $\mathcal{R}$ be any of the following topological Ramsey spaces, with the $\sigma$-closed partial order $\sse^*$: 
\begin{enumerate}
\item 
$\mathcal{R}_{\alpha}$, where 
$1 \leq \alpha < \omega_1$.
\item 
$\mathcal{R}(\mathbb{A})$,  where $\mathbb{A}$ is some generating sequence from   a collection of $\le\om$ many  \Fraisse\  classes with
 the Ramsey property as in \cite{fraisseClasses}.
\item 
$\FIN_{k}^{[\infty]}$, where $k\ge 1$.
\end{enumerate}
Then  $\mathfrak{p}\leq  \mathfrak{p}_{\mathcal{R}}$.
\end{corollary}

\begin{proof}
All of these spaces  satisfy  the IEP.
\end{proof}

Now we show that each $\FIN_{k}^{[\infty]}$ has tower and pseudointersection numbers equal to $\mathfrak{p}$.

\begin{theorem}\label{thm.FINpt}
$\forall k \geq 1,$ $\mathfrak{t}_{\FIN_{k}^{[\infty]}} = \mathfrak{p}_{\FIN_{k}^{[\infty]}}=\mathfrak{p}.$
\end{theorem}

\begin{proof}
Let $\mathcal{F} \subseteq [\omega]^{\omega}$ be a family linearly ordered by $\supseteq^{\ast}$; that is, a tower. 
For  $A \in [\omega]^{\omega}$, write 
$A= \{  a_{0},a_{1},...,a_{n},... \}$ with $a_{n} < a_{n+1}$ for each $n \in \omega$. 
For  $n \in \omega$, define $f_{n}:\omega \longrightarrow \{0,...,k\}$ 
to be the function 
such that $f_{n}(i)=k$
for each $i \in [ a_{n},a_{n+1} )$,  and $f_{n}(i)=0$ for each  $i \notin [a_{n},a_{n+1})$.
Note that the sequence $F_{A}= (f_{n}) _{n \in \omega}$ is a member of $ \FIN _{k}^{[\infty]}$, and moreover,  
 if $A \neq B$ then $F_{A}\neq F_{B}$. 
Furthermore, $A \subseteq^{\ast} B$ implies $F_{A}\leq^{\ast} F_{B}$.  
Let $\mathcal{G}=\{F_{A}: A \in \mathcal{F}\}$ and note that $\mathcal{G}$ is linearly ordered by $\ge^{\ast}$. 
Suppose that $\mathcal{G}$ has a pseudointersection $H=(h_{n})_{n \in \omega}\in  \FIN _{k}^{[\infty]}$. Let $C=\{ \min (\supp (h_{n})):n \in \omega  \}$. 
Note that $C$ is a pseudointersection for $\mathcal{F}$. 
By Lemma \ref{tRleqt}, $\mathfrak{t}_{\FIN_k^{[\infty]}} \leq \mathfrak{t}$.
Corollary \ref{cor.m<pR} 
implies that $\mathfrak{p} \leq \mathfrak{p}_{\FIN_{k}^{[\infty]}}$.
Thus,
it follows that
\begin{equation}
 \mathfrak{t}_{\FIN_{k}^{[\infty]}}\le \mathfrak{t}=\mathfrak{p}\le
  \mathfrak{p}_{\FIN_{k}^{[\infty]}}=\mathfrak{p},
\end{equation}
the middle equality  holding by 
the result  of Malliaris and Shelah in \cite{Malliaris/ShelahJAMS16}.
\end{proof}

The next results will have proofs that follow the outline of Theorem \ref{thm.FINpt} but will involve stronger hypthotheses in order to  apply to the spaces from Subsections \ref{subsec4.1} and \ref{subsec.Fraisse}.

\begin{definition}[ISS$^*$]\label{defn.ISS*}
Let $(\mathcal{R},\le,r)$ be a topological Ramsey space satisfying Independent Sequences of Structures. 
Recall that each finite structure $\mathbb{A}_{i}$  is linearly ordered.
We say that $\mathcal{R}$ satisfies the {\em  ISS$\,^*$} if  
 for all   $k<m$,
  $\mathbb{A}_k$ embeds into $\mathbb{A}_{m}$.
 \end{definition}
  
  It follows from the ISS$^*$ that 
  there are functions $\lambda_{k}$, $k<\om$, such that for each $m\ge k$, $\lambda_k(\mathbb{A}_m)$ is a substructure of $\mathbb{A}_m$ which is isomorphic to $\mathbb{A}_k$.
Moreover, 
for each triple $k<m<n$,
$\lambda_{k}(\mathbb{A}_n)$ is a substructure of $\lambda_{m}(\mathbb{A}_n)$.

\begin{lemma}\label{SFIP}
	Let $\mathcal{F}$ be a family infinite subsets of $\omega$ and $\mathcal{R}$ be a topological Ramsey space with the ISS\,$^*$. 
	Then for each $B \in [\omega]^{\omega}$ there corresponds  a unique  $X_{B} \in \mathcal{R}$ so that given any  $B$, $C \in \mathcal{F}$, the following hold: 
	\begin{enumerate}
		\item $B \neq C$ implies $X_{B} \neq X_{C}$;
		\item $B \subseteq C$ implies $X_{B} \leq X_{C}$;
		\item $B \subseteq^{\ast} C$ implies $X_{B} \leq^{\ast} X_{C}$; and
		\item
		 If $\mathcal{G}=  \{X_{A}: A \in \mathcal{F} \}$ has a pseudointersection, then $\mathcal{F}$ also has a pseudointersection.
	\end{enumerate}
\end{lemma}

\begin{proof}
Given $B \in [\omega]^{\omega},$ let $ \{ b_{0}, b_{1},b_{2},... \}$ be the increasing enumeration of $B$. 
For each $n \in \omega,$ let $X(n)=\lambda_{n}(\mathbb{A}_{b_n})$, and 
  define $X_{B} = \langle X(n): n \in \omega \rangle$. 
Then  $X_{B} \in \mathcal{R}$, since $\mathcal{R}$ satisfies the ISS$^*$.
Moreover, notice that whenever
$B \neq C$ are in $[\omega]^{\omega}$, then 
$X_{B} \neq X_{C}$, so (1) holds.
Suppose that  $B,C \in [\omega]^{\omega}$  satisfy $C \subseteq B$. 
Let $k<\om$ be such that $c_k\in B$.
Then  $c_{k}=b_{m}$ for some $m \geq k$. 
By our construction,
 $X_{C}(k)= \lambda_{k}(\mathbb{A}_{c_k})$ 
and $X_{B}(m)=\lambda_{m}(\mathbb{A}_{b_m})$.
Since $c_k=b_m$,
 $X_{C}(n)$ is a substructure of $X_{B}(m)$.
From these observations, (2) and (3) of the theorem immediately follow.

Fix a family $\mathcal{F} \subseteq [ \omega ]^{\omega}$, and
let $\mathcal{G}= \{ X_{B} : B \in \mathcal{F} \}$. Assume that there exists some $Y \in \mathcal{R}$  which  is a pseudointersection of $\mathcal{G}$. 
We claim that 
\begin{equation}
 D=\{m \in \omega: (\exists i \in \omega)Y(i) \text{ is substructure of } \mathbb{A}_m\}
\end{equation}
 is a pseudointersection of $\mathcal{F}.$ 
 Let $B \in \mathcal{F}$. Since $Y$ is a pseudointersection of $\mathcal{G}$, $Y \leq^{\ast} X_{B}$. 
 Thus, there exists $p< \omega$ such that for every $i >p,$ $Y(i)$ is a substructure of some $X_{B}(j)$ for some $j<\om$. 
 It follows that $D \setminus p \subseteq B$; hence,
 $D\sse^* B$.
\end{proof}

\begin{lemma}\label{tRleqt}
	Let $(\mathcal{R},\le,\le^*,r)$ be a $\sigma$-closed  topological Ramsey space 
	such that for every family $\mathcal{F} \subseteq [ \omega ]^{\omega}$ linearly ordered by $\supseteq^{\ast}$ there exists a family $\mathcal{G} \subseteq \mathcal{R}$ linearly ordered by $\ge^*$ that satisfies the following:
	\begin{enumerate}
		\item $|\mathcal{G}|=|\mathcal{F}|$ and
		\item If $\mathcal{G}$ has a pseudointersection then $\mathcal{F}$ has also a pseudointersection.
	\end{enumerate} Then $\mathfrak{t}_{\mathcal{R}} \leq \mathfrak{t}$ and $\mathfrak{p}_{\mathcal{R}} \leq \mathfrak{p}$.
\end{lemma}
	
\begin{proof}
Let $\mathcal{F} \subseteq [\omega]^{\omega}$ be a family linearly ordered by $\supseteq^{\ast}$ such that $ | \mathcal{F} | < \mathfrak{t}_{\mathcal{R}} $. By hypothesis there is a family $\mathcal{G} \subseteq \mathcal{R}$ linearly ordered by $\ge^*$ that satisfies that $|\mathcal{G}|=|\mathcal{F}|$. So  $ | \mathcal{G} | = | \mathcal{F} | < \mathfrak{t}_{\mathcal{R}} $. Therefore $\mathcal{G}$ has a pseudointersection and by hypothesis  2), $\mathcal{F}$ has also a pseudointersection. Hence $\mathfrak{t}_{\mathcal{R}} \leq \mathfrak{t}$. A similar argument proves that $\mathfrak{p}_{\mathcal{R}} \leq \mathfrak{p}$.	
\end{proof}	


\begin{theorem}\label{pRtR}
Let $\mathcal{R}$ be a topological Ramsey space that satisfies  the ISS$^*$.
 Then $\mathfrak{p}_{\mathcal{R}}= \mathfrak{t}_{\mathcal{R}}=\mathfrak{p}.$
\end{theorem}

\begin{proof}
Suppose $\mathcal{R}$ satisfies the  ISS$^*$.
Then by Lemmas \ref{SFIP} and \ref{tRleqt},
$\mathfrak{t}_{\mathcal{R}} \leq \mathfrak{t}$. 
 Note  that $\mathcal{R}$ satisfies 
 the IEP, since this follows from the ISS$^*$.
then  Theorem \ref{pleqpR} implies that 
  $\mathfrak{m}(\sigma$-centered)$\leq \mathfrak{p}_{\mathcal{R}}$. 
Since $\mathfrak{p}= \mathfrak{m}(\sigma$-centered),  we have $\mathfrak{p} \leq \mathfrak{p}_{\mathcal{R}}\leq \mathfrak{t}_{\mathcal{R}}\leq \mathfrak{t}.$
The equality follows from the result of Malliaris and Shelah, that $\mathfrak{p}=\mathfrak{t}$.
\end{proof}

\begin{corollary}\label{cor.ptRal}
\begin{enumerate}
\item 
For all $1 \leq \alpha < \omega_1$, $\mathfrak{t}_{\mathcal{R}_{\alpha}} = \mathfrak{p}_{\mathcal{R}_{\alpha}} = \mathfrak{p}.$
\item 
If $\mathcal{R}$ is a topological Ramsey space generated by a collection of Fra\"{i}ss\'{e} classes with the Ramsey property, then $\mathfrak{t}_{\mathcal{R}} = \mathfrak{p}_{\mathcal{R}} = \mathfrak{p}.$
\end{enumerate}
\end{corollary}

\begin{proof}
The topological Ramsey space in the hypothesis  satisfy the ISS$^*$, so the corollary follows from 
Theorem  \ref{pRtR}.
\end{proof}

In particular, $\forall n \geq 1$, 
the pseudointersection number and tower number for the $n$-hypercube space $\mathcal{H}^{n}$ all equal
$\mathfrak{p}$, since these are special cases of (2) in Corollary \ref{cor.ptRal}.

\subsection{Pseudointersection and tower numbers for the forcings $\mathcal{P}(\om^{\al}/\Fin^{\otimes\al})$}\label{subsec.ptE_al}

Next, we look at the pseudointersection and tower numbers for the high dimensional Ellentuck spaces $\mathcal{E}_{\al}$, for $2\le \al<\om_1$.
Recall that $(\mathcal{E}_{\al},\sse^{*\al})$ is forcing equivalent to $\mathcal{P}(\om^{\al})/\Fin^{\otimes\al}$.
Hence, for the high and infinite dimensional Ellentuck spaces, the partial order $\le^*$ denotes $\sse^{*\al}$.
We point out that  for the spaces $\mathcal{E}_{\al}$, the $\sigma$-closed partial order defined by Mijares  (recall Definition \ref{defn.Mijares*}) is intermediate between $\le$ and $\sse^{*\al}$ and hence  produces the same separative quotient.

The following theorem is proved in \cite{Szymanski/Zhou81}.

\begin{theorem}[Szyma{\'{n}}ski  and Zhou, \cite{Szymanski/Zhou81}]\label{thm.SZ}
	$\mathfrak{t}(\Fin \otimes \Fin)= \omega_{1}$.
\end{theorem}

\begin{proposition}
	$\mathfrak{p}_{\mathcal{E}_{2}} = \mathfrak{t}_{\mathcal{E}_{2}}= \omega_{1}$.
\end{proposition}
\begin{proof}
This is a consequence of the last two theorems.
\end{proof}

In what follows, we 
show that for each $\al\in [2,\om_1)$,  the pseudointersection and tower numbers of $\mathcal{E}_{\al}$ are equal to $\om_1$.
In fact, this is true for each space $\mathcal{E}_B$ in \cite{DobrinenJML16}, where $B$ is a uniform barrier of infinite rank. 
We
point out that  for $2\le k<\om$, the following results were found by Kurili\'{c} in \cite{Kurilic15}, though we were unaware of those results at the time that our results were found.
It is important to note that the forcings  $\mathbb{P(\al)}$ in 
\cite{Kurilic15} are different from the forcings $\mathcal{P}(\al)/\Fin^{\otimes\al}$, so  for infinite countable ordinals,  the results below are new.

\begin{notation} For every $k \geq 2,$ $1< l < k,$ $X \in \mathcal{E}_{k}$ and $x \in \ome{k}:$
\begin{enumerate}
\item Let $\max  x$ denote the last member of the finite sequence $x$.
\item Let $\pi_{1}(X)$ denote the set $\{x_{0}: x\in X\}$.
\item Denote by $\pi_{l}(X)$ the set $\{x\upharpoonright l: x \in X\}.$
\end{enumerate}
\end{notation}

Note that because of the definition of $\mathcal{E}_{k}$ spaces, $\pi_{l}(X)\in \mathcal{E}_{l}.$
\begin{definition}
Let $X$ be a member of $\mathcal{E}_{k}$ and $s$ a finite aproximation of $X$. Write $\pi_{1}[X]$ as an increasing sequence $\{n_{0},n_{1},...,n_{j},...\}$. We will say that $s$ is the i-th \emph{full finite approximation} of $X$ if $s$ is the $\sqsubseteq$-least finite approximation of $X$ such that there exists $x \in s$ such that $\min x = n_{i}$. We will denote by $\mathfrak{a}^{k}_{i}(X)$ the i-th full finite approximation of $X$.
\end{definition}
Note that for every $i \in \omega$, if $x$ is the $\prec$-least member of $\mathfrak{a}^{k}_{i}(X) / \mathfrak{a}^{k}_{i-1}(X) $ then $x_{0}=n_{i}$. Also note that for every $x \in X$ such that $x_{0} = n_{i}$, there exists $Y \leq X$ such that $x$ is the $\prec$-least member of $\mathfrak{a}^{k}_{i}(Y) / \mathfrak{a}^{k}_{i-1}(Y) $.

\begin{lemma}\label{lem.11}
	For every $k \in \omega$ and $X \in \mathcal{E}_{k}$, there exists a $X' \in \mathcal{E}_{k+1}$ such that for every $X, Y \in \mathcal{E}_{k}$ and $Z \in \mathcal{E}_{k+1}$:
	
	\begin{enumerate}
		\item $\pi_{k}(X')=X$,
		\item $Y \leq X$ implies $Y' \leq X',$
		\item $Y \leq^{\ast} X$ implies $Y' \leq^{\ast} X'$ and
		\item $Z \leq ^{\ast} X'$ implies $\pi_{k}(Z) \leq ^{\ast} X$.
	\end{enumerate}
\end{lemma}

\begin{proof}
	The proof will be by induction on $k \in \omega$. Fix $k=2$. Let $X$ be a member of $\mathcal{E}_{2}$, write $X = \{ x_{0},x_{1},...,x_{i},... \},$ with $x_{i} \prec x_{i+1}$ for every $i \in \omega.$ Write $\pi_{1}(X)=\{n_{m}:m \in \omega\}$ where $n_{m}< n_{m+1}$ for every $m \in \omega$. We want to construct a member of $ \mathcal{E}_{3} $ such that $\pi_{2}[X']=X$. We will construct such $X'$ step by step by extending members of full finite approximations. Note that $\mathfrak{a}^{2}_{1}(X)=\{x_{0}\}$. In this case we extend $x_{0}$ with $\bar{x}^{0}_{0}= {x_{0}} ^{\frown} \max x_{0}$. Now fix $i>0$. Let $l_{i}= \sum_{j \leq i+1}j$, then $\mathfrak{a}^{2}_{i}(X)=\{ x_{0},x_{1},...,x_{l_{i}-1} \}$ and $\mathfrak{a}^{2}_{i}(X) / \mathfrak{a}^{2}_{i-1}(X)=\{  x_{l_{i}-i-1},...,x_{l_{i}-1}\} $. Note that for every $m \in [0,i+1)$, $\pi_{1}(x_{l_{i}-i-1+m})=n_{m}$. Given $x \in \mathfrak{a}^{2}_{i}(X)$, there exists $m \in [0,i+1)$ such that $\pi_{1}(x)=n_{m}$. We extend $x$ to $\bar{x}^{i}=x ^{\frown} (\max x_{l_{i}-i-1+m})$. Note that for every $i \in \omega$ and $j \in [0,l_{i})$, $\pi_{2}(\bar{x}^{i}_{j})=x_{j}$. Let $X'=\{ \bar{x}^{i}_{j}: i \in \omega, j \in [0,l_{i}) \}$. Note that by the definition of $X'$, $X' \in \mathcal{E}_{3}$ and $\pi_{2}(X')=X$. We will prove that if $Y \leq X$ then $Y' \leq X'$. Take $y \in Y'$ then there are $ j,k \in \omega $ with $j \leq k$ such that $y= {y_{j}} ^{\frown} ( \max y_{k} ) $. Note that by construction, $\pi_{1}(y_{j})=\pi_{1}(y_{k})$. Since $Y \leq X$, there are some $j',k' \in \omega$ with $j'<k'$ such that $x_{j'}=y_{j} $ and $x_{k'}=y_{k}$. There exists an $i \in \omega$ such that $x_{k'} \in \mathfrak{a}^{2}_{i}(X) / \mathfrak{a}^{2}_{i-1}(X)$. Then $\bar{x}^{i}_{j'}={x_{j'}} ^{\frown} (\max x_{k'}) = y$ and $\bar{x}^{i}_{j'} \in X'$. So, $Y' \leq X'$. By similar arguments, $Y \leq^{\ast} X$ implies $Y' \leq^{\ast} X'.$  Now suposse that $Z \leq ^{\ast} X'$, then $| X' / Z| < \omega$. Note that if $| X / \pi_{2}(Z)| > \omega$, by the construction of $X'$, $| X' / Z| > \omega$. Therefore $\pi_{2}(Z) \leq^{\ast} X$.\\

	Let $k>2$ be a natural number. Now assume we have the conclusion for every member of $[2,k]$ and we want to prove it for $k+1$. Fix $X \in \mathcal{E}_{k+1}$ and write $\pi_{1}(X)=\{ n_{0},n_{1},...,n_{m},... \},$ with $n_{m}< n_{m+1}$ for every $m \in \omega$. Write $X(m)$ to denote the subtree of $X$ such that for every $x \in X(m)$, $\pi_{1}(x)=n_{m}$. Note that every $X(m)$ contains a member of $\mathcal{E}_{k}.$ Now, for every $m \in \omega$ let $X(m) \upharpoonright (0,...,k]$ denote the collection of $x \upharpoonright (0,...,k]$ where $x$ is a member of $X(m)$. Note that $X(m) \upharpoonright (0,...,k] \in \mathcal{E}_{k}$. By induction hypothesis, for every $m \in \omega$ we can extend $X(m) \upharpoonright (0,...,k]$ to a $X(m)' \in \mathcal{E}_{k+1}$ such that $\pi_{k}(X(m)')= X(m) \upharpoonright (0,...,k] $. Let $X'= \bigcup_{m \in \omega} \{ (n_{m})^{\frown} x : x \in X(m)' \}$. Note that $X' \in \mathcal{E}_{k+2}$. Since for every $m \in \omega$, $\pi_{k}(X(m)')= X(m) \upharpoonright (0,...,k] $, $\pi_{k+1}(X')=X$. Take $X,Y \in \mathcal{E}_{k+1}$ such that $Y \leq X$. Take $y \in Y'$ then exists $n_{m} \in \omega$ such that $y = (n_{m}) ^{\frown} z$ with $z \in  Y(m)'.$ Then $(z\upharpoonright (0,k]\in Y(m)$ and since $Y \leq X$ there exists an $l \in \omega$ such that $Y(m) \subset X(l).$ Hence by hypothesis induction $Y(m)' \leq X(l)'$  and $y \in X'.$ By similar arguments, $Y \leq^{\ast} X$ implies $Y' \leq^{\ast} X'.$ Now suposse that $Z \leq ^{\ast} X'$, then $| X' / Z| < \omega$. Note that if $| X / \pi_{k+}(Z)| > \omega$, by the construction of $X'$, $| X' / Z| > \omega$. Therefore $\pi_{k+1}(Z) \leq^{\ast} X$.
\end{proof}

\begin{proposition}\label{prop.12}
	For every $k \in \omega$ such that $k > 1$, $\mathfrak{t}_{\mathcal{E}_{k+1}} \leq \mathfrak{t}_{\mathcal{E}_{k}} .$
\end{proposition}

\begin{proof}
 Fix $k\in \omega$.  Let $\kappa < \mathfrak{t}_{\mathcal{E}_{k+1}}$ be a cardinal, we will prove that $\kappa < \mathfrak{t}_{\mathcal{E}_{k}}$. Let $\mathcal{F} \subseteq \mathcal{E}_{k}$ be a family linearly ordered by $\subseteq^{\ast}$. By the last Lemma for every $X \in \mathcal{E}_{k}$ there exists an $X' \in \mathcal{E}_{k+1}$ with properties $1$,$2$ and $3.$ Let $\mathcal{G}=\{X':X \in \mathcal{F}\}$ and note that $\mathcal{G}$ is linearly ordered by $\leq^{\ast}$. Since $\kappa < \mathfrak{t}_{\mathcal{E}_{k+1}}$, $\mathcal{G}$ has a pseudointersection $Z$. By 3 of the last Lemma, $\pi_{k}(Z)$ is a pseudointersection of $\mathcal{F}$. Then $\kappa \leq \mathfrak{t}_{\mathcal{E}_{k}}$. Therefore $\mathfrak{t}_{\mathcal{E}_{k+1}} \leq \mathfrak{t}_{\mathcal{E}_{k}} .$
\end{proof}

\begin{theorem}
	For every $k > 2$, $\mathfrak{t}_{\mathcal{E}_{k}}=\mathfrak{p}_{\mathcal{E}_{k}}= \omega_{1}$.
\end{theorem}

\begin{proof}
Let $k>2$ be given. Since $(\mathcal{E}_{k}, \leq^{\ast})$ is a $\sigma$-closed partial order $\omega_{1} \leq \mathfrak{t}_{\mathcal{E}_{k}}$. By the last proposition,  $\mathfrak{p}_{\mathcal{E}_{k}} \leq \mathfrak{t}_{\mathcal{E}_{k}} \leq \mathfrak{t}_{\mathcal{E}_{2}}=  \omega_{1}$. Therefore $\mathfrak{t}_{\mathcal{E}_{k}}=\mathfrak{p}_{\mathcal{E}_{k}}= \omega_{1}$.
\end{proof}

In fact, by a similar proof, we obtain the following:

\begin{theorem}
Let $B$ be a uniform barrier of  countable rank at least $2$.
Then $\mathfrak{t}_{\mathcal{E}_B} \leq \mathfrak{t}_{\mathcal{E}_{2}} .$
\end{theorem}

\begin{proof}
Similarly to the proof of 
Lemma \ref{lem.11}, in fact, for each $X\in\mathcal{E}_2$, there is an $X'\in\mathcal{E}_B$ such that  (1) -- (4) of that lemma hold, where $k=2$.
Then similarly to 
Proposition \ref{prop.12},
it follows that $\mathfrak{t}_{\mathcal{E}_B} \leq \mathfrak{t}_{\mathcal{E}_{2}}$.
\end{proof}

\subsection{The Carlson-Simpson Space}\label{subsec.CS}

We finish this section with a proof that the Carlson-Simpson space also has tower number equal to $\om_1$.
Recall from Subsection \ref{subsec.CS} the 
 Carlson-Simpson space 
 $\mathcal{E}_{\infty}$, which   is not related with high dimensional Ellentuck spaces.
 The following proposition is a consequence of a Proposition from Carlson that appears in \cite{matet1986}, where Carlson prove that there is a family 
 $\{X_{\alpha}:\alpha< \omega_1\}\subseteq \mathcal{E}_{\infty}$ such that there are not $X \in \mathcal{E}_{\infty}$ such that for every $\alpha< \omega_1,$ $X \leq ^{\ast}X_\alpha$.
 
\begin{proposition}
$ \mathfrak{p}_{\mathcal{E}_{\infty}}= \mathfrak{t}_{\mathcal{E}_{\infty}}= \omega_1.$
\end{proposition}

This is interesting in light of the following proposition.

\begin{proposition}
The partial order $(\mathcal{E}_{\infty},\leq^{\ast})$ forces a Ramsey ultrafilter.
\end{proposition}

\begin{proof}
Let $G$ be an $(\mathcal{E}_{\infty},\leq^{\ast})$-generic filter,
and 
let $\mathcal{U}$ be the ultrafilter on $\om$ generated by sets $p(E)$ with $E \in G$.
Suppose we are given a coloring $c:[\omega]^{2} \rightarrow 2$. 
 Define $\mathcal{D}= \{ F \in \mathcal{E}_{\infty}: | c [ p(F)]^{2} |=1 \}$. We will prove that $\mathcal{D}$ is a dense subset of $\mathcal{E}_{\infty}$. 
 Let $E \in \mathcal{E}_{\infty}$ be given. By Ramsey's Theorem, there exists $M \in [p(E)]^{\omega}$ such that $M$ is monochromatic.
  Define a rigid surjection $h: \omega \rightarrow \omega$  recursively as follows. Let  $h(0)=f_{E}(0)=0$. 
  Now, fix $i \ge 1$ and suppose we have defined $h(j)$ for each $j < i$.  There are three cases:
  If there exists $j<i$ such that 
  $f_{E}(i)=f_{E}(j)$, then let $h(i)=h(j)$. 
  Otherwise, for all $j<i$, 
  $f_{E}(i)\ne f_{E}(j)$.
If $i \in M$, let $h(i)=\max \{h(j):j<i\}+1$.
If $i\not\in M$,
let 
  $h(i)=0$. 
  Let $F=E_h$, the partition of $\omega$ generated by $h$. 
  Note that $F \in \mathcal{E}_{\infty}$ and $p(F) = M$. Since $M$ is monochromatic, $F \in \mathcal{D}$. Therefore $\mathcal{D}$ is a dense subset of $(\mathcal{E}_{\infty}, \leq ^{\ast})$. Since $G$ is a generic filter, there exists $H \in \mathcal{D} \cap G$. Then $p(H)$ is monochromatic for the coloring $c$.
\end{proof}

In fact, Navarro Flores has proved recently that every topological Ramsey space forces a Ramsey ultrafilter, answering a question of the first author. 
This result will appear in a forthcoming paper.

\section{Open problems}\label{sec:questions}

The results in Sections 5 and 6   may be used to calculate more Ramsey degrees of ultrafilters. 
Especially of interest are the further degrees $t(\mathcal{E}_k,n)$ for $n\ge 3$, for any $k\ge 2$.

As we  have seen in the previous  section, for several  different classes of  topological Ramsey spaces,
 their pseudointersection number is equal to the respective tower number.
  Also, we know that for  some  topological Ramsey spaces, $\mathcal{R}$, $\mathfrak{p}_{\mathcal{R}}=\mathfrak{t}_{\mathcal{R}}= \mathfrak{p}$ and for others, $\mathfrak{p}_{\mathcal{R}}=\mathfrak{t}_{\mathcal{R}}= \omega_1.$
We conclude this paper with the most pressing open problems in this subject.

\begin{question}
Is $\mathfrak{p}_{\mathcal{R}} \leq \mathfrak{p}$ for every topological Ramsey space $\mathcal{R}$?
\end{question}

\begin{question}
Is there a topological Ramsey space with pseudointersection number different from its tower number?
\end{question}

\begin{question}
Is there a topological Ramsey space  with pseudointersection number different from both $\omega_{1}$ and $\mathfrak{p}?$
\end{question}

\begin{acknowledgements}Dobrinen was supported by National Science Foundation Grants DMS-14247 and DMS-1600781. Navarro Flores was supported by CONACYT, and Dobrinen's National Science Foundation Grants DMS-14247 and DMS-1600781.
\end{acknowledgements}

\bibliographystyle{spmpsci}      
\bibliography{refs}   

%
%

\end{document}